\documentclass[11pt]{amsart}

\usepackage{amsmath,amsfonts,amssymb,amsthm,mathtools}

\usepackage{verbatim}
\usepackage{amsmath,amsfonts}
\usepackage[mathscr]{euscript} 
\usepackage{amsthm}
\usepackage{blkarray,bigstrut} 
\usepackage{url}

\usepackage{graphicx} 

\usepackage{enumitem} 

\usepackage{hyperref} 
\hypersetup{colorlinks} 

\usepackage[colorinlistoftodos]{todonotes}

\RequirePackage{cleveref}
\usepackage{hypcap}
\hypersetup{colorlinks=true, citecolor=darkblue, linkcolor=darkblue}
\definecolor{darkblue}{rgb}{0.0,0,0.7}
\newcommand{\darkblue}{\color{darkblue}}

\definecolor{darkred}{rgb}{0.68,0,0}
\newcommand{\darkred}{\color{darkred}}

\definecolor{darkgreen}{rgb}{0,.38,0}
\newcommand{\darkgreen}{\color{darkgreen}}

\newcommand{\defn}[1]{\emph{\darkblue #1}}
\newcommand{\defna}[1]{\emph{\darkred #1}}
\newcommand{\defnb}[1]{\emph{\darkblue #1}}
\newcommand{\defng}[1]{\emph{\darkgreen #1}}

\setlist[enumerate]{
	label=\textnormal{({\roman*})},
	ref={\roman*}}

\makeatletter
\def\th@plain{%
	\thm@notefont{}
	\itshape 
}
\def\th@definition{%
	\thm@notefont{}
	\normalfont 
}
\makeatother

\makeatletter
\def\fdsy@scale{1}
\newcommand\fdsy@mweight@normal{Book}
\newcommand\fdsy@mweight@small{Book}
\newcommand\fdsy@bweight@normal{Medium}
\newcommand\fdsy@bweight@small{Medium}

\DeclareFontFamily{U}{FdSymbolB}{}
\DeclareFontShape{U}{FdSymbolB}{m}{n}{
	<-7.1> s * [\fdsy@scale] FdSymbolB-\fdsy@mweight@small
	<7.1-> s * [\fdsy@scale] FdSymbolB-\fdsy@mweight@normal
}{}
\DeclareFontShape{U}{FdSymbolB}{b}{n}{
	<-7.1> s * [\fdsy@scale] FdSymbolB-\fdsy@bweight@small
	<7.1-> s * [\fdsy@scale] FdSymbolB-\fdsy@bweight@normal
}{}

\makeatother

\DeclareSymbolFont{fdrelations}{U}{FdSymbolB}{m}{n}
\SetSymbolFont{fdrelations}{bold}{U}{FdSymbolB}{b}{n}

\DeclareMathSymbol{\lescc}{\mathrel}{fdrelations}{66}

\newtheorem{thm}{Theorem}[section]

\newtheorem{claim}[thm]{Claim}
\newtheorem{cor}[thm]{Corollary}
\newtheorem{prop}[thm]{Proposition}

\newtheorem{conv}[thm]{Convention}

\theoremstyle{definition}

\newtheorem{rem}[thm]{Remark}

\numberwithin{figure}{section}
\numberwithin{equation}{section}

\def\emp{\nothing}

\def\nn{\mathbb N}

\def\rr{\mathbb R}

\def\qqq{\mathbb Q}

\def\sm{\smallsetminus}

\def\Om{\Omega}

\def\la{\lambda}

\def\al{\alpha}
\def\be{\beta}

\def\vp{\varphi}

\def\cC{\mathcal C}

\def\cA{\mathcal A}

\def\cF{\mathcal F}

\def\cL{\mathcal L}

\def\cP{\mathcal P}
\def\cQ{\mathcal Q}

\def\ssu{\subset}

\def\<{\langle}
\def\>{\rangle}

\def\rK{{\text{\sc {K}}}}

\def\0{{\mathbf 0}}

\def\LL{{\mathbb{L}}}

\def\nothing{\varnothing}

\def\.{\hskip.06cm}
\def\ts{\hskip.03cm}

\def\bba{\textbf{\textit{a}}}

\def\bbz{\textbf{\textit{z}}}

\def\eeb{{\text{\bf e}}}

\def\di{{\small{\ts\diamond\ts}}}

\def\ze{{\zeta}}

\newcommand{\maj}{\mathrm{maj}}
\newcommand{\PP}{\operatorname{PP}}
\newcommand{\SPP}{\operatorname{SPP}}

\newcommand{\RPP}{\operatorname{RPP}}
\newcommand{\SSYT}{\operatorname{SSYT}}
\newcommand{\SYT}{\operatorname{{\rm SYT}}}

\newcommand{\ges}{\operatorname{{\geqslant}}}

\def\nin{\noindent}

\def\SP{{\textsc{\#P}}}

\def\SP{{\textsc{\#{}P}}}

\def\aF{\textrm{F}}
\def\aFr{\textrm{\em F}}

\def\leqs{\leqslant}

\def\A{A}

\def\fe{f}

\def\RGF{{{\cF}}}

\DeclareMathOperator{\Ec}{\mathcal{E}} 

\DeclareMathOperator{\Rb}{\mathbb{R}}

\def\qb{\textbf{\textit{q}}{}}

\newcommand{\en}{{f}}

\def\rb{\textbf{\textit{r}}\hskip-0.03cm{}}

\DeclareMathOperator{\lqr}{\langle\qb{}\., \rb \. \rangle}
\DeclareMathOperator{\db}{\mathbf{d}}

\title[Multivariate correlation inequalities]{Multivariate correlation inequalities for $P$-partitions}

\date{\today}

\author[Swee Hong Chan \and Igor Pak]{\ Swee Hong Chan$^{\star}$ \ \. \and \ \. Igor~Pak$^{\di}$}

\thanks{\thinspace ${\hspace{-.45ex}}^\star$Department of Mathematics,
Rutgers University, Piscataway, NJ, 08854.
 \.  Email: \ts \texttt{sweehong.chan@rutgers.edu}}

\thanks{\thinspace ${\hspace{-.66ex}}^\di$Department of Mathematics,
UCLA, Los Angeles, CA, 90095. \.  Email: \ts \texttt{{pak}@math.ucla.edu}}

\begin{document}

\begin{abstract}
Motivated by the Lam--Pylyavskyy inequalities for Schur functions,
we give a far reaching multivariate generalization of Fishburn's
correlation inequality for the number of linear extensions of posets.
We then give a multivariate generalization of the Daykin--Daykin--Paterson
inequality proving log-concavity of the order polynomial of a poset.
We also prove a multivariate $P$-partition version of the cross-product
inequality by Brightwell--Felsner--Trotter.   The proofs are based on
a multivariate generalization of the Ahlswede--Daykin inequality.
\end{abstract}

	\maketitle

\section{Introduction} \label{s:intro}

Arguably, \defng{linear extensions} play as much a central role in poset theory
as \defng{standard Young tableaux} \ts in algebraic combinatorics.  While the former
combinatorial objects obviously generalize the latter, this connection is yet
to be fully explored.  In fact, the development in the two areas seem to move
along parallel tracks as we explain below.

The story of this paper is an interplay between these two areas of combinatorics,
which makes both the motivation and presentation of the results
somewhat less accessible.  To mitigate this, we include two separate (and almost
completely non-overlapping) versions of the introduction addressing audiences
with different background (see also~$\S$\ref{ss:finrem-origin}).

The results themselves are postponed to later sections and assume fluency
in both areas.  While the reader may choose to read only the results
that are closer to their interests, reading both sides of the story
can enhance the experience.  To help navigate between the areas,
we include detailed notation and some background in Section~\ref{s:notation}.

\smallskip

\subsection*{Poset theoretic perspective}
Our first result (Theorem~\ref{t:main-LE}) is a self-dual generalization of
the remarkable \defna{Fishburn's correlation inequality} \ts (Theorem~\ref{t:Fish})
for the numbers of linear extensions of poset order ideals.
We  further extend it  to a correlation inequality for order polynomials,
and then even further to their $q$-analogues and multivariate $\qb$-analogues
(Theorems~\ref{t:main-OP-multi} and~\ref{t:main-K}).
To understand the proofs it is worth examining the historical background
and motivation behind earlier results.  

Following up on the works by
Harris (1960) and Kleitman (1966), Fortuin--Kasteleyn--Ginibre introduced
the celebrated \defn{FKG inequality} \ts \cite{FKG}.  This correlation inequality
was further generalized in a series of papers, most notably by Ahlswede--Daykin
\cite{AD}, who proved a very general \defn{AD inequality} \ts (Theorem~\ref{thm:AD}),
which is also called the \defng{four functions theorem} \cite[$\S$6.1]{AS}.
This result is so general that it has an elementary albeit somewhat involved
proof by induction (ibid.).  For the many followup investigations of
correlation inequalities, see e.g.\ \cite[$\S$15]{AB}, \cite[$\S$5]{Pak-what}, and
earlier overviews in \cite{FS,Gra,Win}.

In a direct application to posets, Shepp \cite{She} was able to use the FKG
inequality and a clever limit argument to prove the \defn{XYZ inequality} 
(see e.g.\ \cite[$\S$6.4]{AS}),
the most remarkable correlation inequality for linear extensions of posets, conjectured
earlier by Rival and others.  This brings us to Fishburn \cite{Fish1},
who established the \defn{Fishburn's inequality} \ts (Theorem~\ref{t:Fish})
as a tool in his proof of the strict version of the XYZ inequality.
We note that Shepp's limit argument does not imply  the strict version,
so Fishburn's proof uses the AD inequality instead.

Motivated by enumerative applications and Fishburn's work, Bj\"orner \cite{Bjo11}
proved the \defn{$q$-FKG inequality} \ts generalizing the FKG inequality.
Christofides \cite{Chri} then found the \defn{$q$-AD~inequality}, answering
Bj\"orner's question.  In a joint work with Panova \cite{CPP},
we employed Bj\"orner's $q$-FKG inequality to obtain $q$-analogues of
inequalities for order polynomials of interest in enumerative combinatorics.

In our most recent paper \cite{CP3}, we find several correlation inequalities
whose proof required the \emph{combinatorial atlas} technique and does not have
a natural $q$-analogue.  Among other results, we proved a series of upper bounds
on correlation inequalities (when they are written in the form of a ratio $\ge 1$),
in some cases serving as a counterpart to the Fishburn's inequality.

The generality of our upper bounds in \cite{CP3} and the self-dual nature
of related results on Young tableaux naturally leads to our self-dual generalization of
Fishburn's inequality.  Just like the original proofs by Shepp and Fishburn,
our proof is via posets' order polynomials, which naturally arise in this
setting.  Curiously, to prove our main theorem (Theorem~\ref{t:main-OP-multi}),
we use a multivariate generalization (Theorem~\ref{thm:multi-q-AD}) of Christofides's
$q$-AD inequality.

At this point one would want to compare our results (notably Theorem~\ref{t:main-K}),
to those by Lam and Pylyavskyy \cite{LP07}, which are closely related and partly
inspired this paper.  They also prove a multivariate correlation inequality
for order preserving maps on posets, which in some cases coincides with ours
(cf.\ Corollary~\ref{c:LP-RPP2} and Remark~\ref{r:LP-gen}).  Unfortunately,
their meet and join operations on order ideals are noncommutative and are
therefore distinct from the more traditional definitions that we use.  Thus,
while the results in \cite{LP07} might appear similar and even more general
at a first glance (partially because they use the same notation),
in full generality the similarity is misleading.

Now, Lam--Pylyavskyy's \defng{Cell Transfer Theorem} \cite[Thm~3.6]{LP07} has
a more general setting given by certain functions on poset's Hasse diagram.
When it comes to skew Young diagrams, this allows the authors to recover
the same reverse plane partitions results that we do, as well as semistandard
Young tableaux results.  We also recover their correlation inequality for
Schur functions by making additional arguments (Section~\ref{s:Schur}).

To summarize the comparison, neither result implies the other.
Our meet and join notions are more standard leading to a self-dual
generalization of Fishburn's inequality using a more standard tool
(generalized AD inequality). On the other hand, the Lam--Pylyavskyy's
ad hoc definitions allow them to recover the same Young tableaux results
with an advantage of their proof giving an explicit combinatorial
injection (cf.~$\S$\ref{ss:finrem-SP}).

We give two applications of the multivariate AD~inequality to poset
inequalities.  First, we prove a multivariate cross-product inequality
for order preserving maps on posets (Theorem~\ref{thm:CPC-OP}),
giving a variation on the \defng{cross-product inequality} \ts by
Brightwell--Felsner--Trotter \cite{BFT}.  This result is new
even for the usual (unweighted) setting.  Note that the
cross-product inequality remains a conjecture in full
generality (Remark~\ref{r:CPC}).

Finally, we give a multivariate extension of the \defn{Daykin--Daykin--Paterson
\ts {\rm (DDP)} inequality} (Theorem~\ref{t:DDP}), which was originally
conjectured by Graham in \cite{Gra},
and proved in \cite{DDP} by an ingenuous direct injection.\footnote{This
injection eluded us in the first version \cite{CPP}, when we were not aware
of \cite{DDP} and proved an asymptotic version of the DDP~inequality
which we called \emph{Graham's conjecture}.}
In fact, Graham originally suggested that the DDP inequality could be
proved by the AD~inequality (see Remark~\ref{r:Graham}).  We provide
such a proof in~$\S$\ref{ss:DDP-old}.  Then, motivated by the structure
of the multivariate AD~inequality, we give a multivariate generalization
of the DDP inequality (Theorem~\ref{t:DDP-multi}).   We conclude with
a multivariate log-concavity of the order polynomial
(Corollary~\ref{c:DPP-multi}), generalizing our recent joint
result with Panova \cite{CPP}.

\smallskip

\subsection*{Algebraic combinatorics perspective}
Our main result is a generalization of the remarkable \defna{Lam--Pylyavskyy
correlation inequality} \ts (Theorem~\ref{t:P-Schur}) for Schur functions
and reverse plane partitions to a self-dual (multivariate) correlation
inequalities for general posets (Theorems~\ref{t:main-OP-multi}
and~\ref{t:main-K}).  Specializations of our main result give
correlation inequalities for $q$-analogues of the number standard Young
tableaux for both straight and skew shapes, which generalize Bj\"orner's
inequality (Corollary~\ref{c:Bjo}).

To understand the proofs it is worth examining the historical background
and motivation behind earlier results.  The study of inequalities for the
symmetric functions goes back to Newton (1707), who proved the
log-concavity \. $\eeb_k^2 \, \ge \, \eeb_{k+1} \. \eeb_{k-1}$ \.
of elementary symmetric polynomials \. $\eeb_k(x_1,\ldots,x_n)$, for all \ts $x_i \in \rr$.
We refer to \cite{Mac,EC} for a thorough treatment of symmetric functions.

Over the past century, symmetric functions have received a great deal of attention
due to their connections and applications in representation theory, as well as a
host of other fields (enumerative algebraic geometry, integrable probability, etc.)
With many identities came inequalities, which were often proved by tools from other
areas.  We refer to \cite{Bre,Bre2,Sta2} for somewhat dated surveys
and to \cite{Bra,Huh} for a more recent overviews of positivity results.

Some recent highlights include inequalities for values of Schur functions
conjectured by Cuttler--Greene--Skandera \cite{CGS} and proved by Sra~\cite{Sra},
the log-concavity of normalized Schur polynomials by
Huh--Matherne--M\'esz\'aros--St.~Dizier \cite{HMMS}, and the
\emph{Schur positivity correlation inequality} \ts
by Lam--Postnikov--Pylyavskyy \cite{LPP07} (see Remark~\ref{r:LPP}).

Building on the ideas which go back to MacMahon (1915), Stanley introduced in
his thesis \cite{Sta-PP} the \defn{$P$-partition theory}, which is closely
related to the study of the order polynomial of posets, and to the \defn{major index}
\ts statistics on linear extensions \cite[$\S$3.15]{EC}.  Motivated by applications
to plane partitions, the study of $P$-partitions became an important subject of
its own.  The order polynomial of a poset turned out to coincide with the
Ehrhart polynomial of the order polytope (see e.g.\ \cite[$\S$4.6.2]{EC}).

The Lam--Pylyavskyy paper \cite{LP07} uses Stanley's $P$-partition theory to
obtain inequalities for the numbers of $P$-partitions with multivariate weights.
The authors presented an explicit combinatorial injection called the \defng{cell
transfer}, which proves inequalities in a very general setting.
As the main application they succeeded in establishing the monomial positivity
correlation inequality for Schur functions (Theorem~\ref{t:P-Schur}),
which was soon overshadowed by the stronger Schur positivity LPP correlation
inequality mentioned above. Their approach also extends to monotonicity of
quasisymmetric functions which arise from $P$-partitions \cite{LP08}.

In this paper, we take the core part of the Lam--Pylyavskyy general inequality
and generalize it in the direction which is more natural from the poset theoretic point
of view (Theorem~\ref{t:main-K}).  Since multivariate inequalities are uncommon in poset
theory, we give a multivariate extension of the AD~inequality, an important tool in the area.
We then show that our multivariate extension is strong enough to also imply the
above mentioned Lam--Pylyavskyy's monomial positivity.

Finally, we show that this multivariate approach can be used to prove new inequalities
for general posets. Notably, we prove a new cross-product inequality
(Theorem~\ref{thm:CPC-OP}), and extend DDP and CPP log-concave inequalities
for general posets (Theorem~\ref{t:DDP-multi} and Corollary~\ref{c:DPP-multi}).

\smallskip

\subsection*{Paper structure}
We start with a lengthy Section~\ref{s:notation} with the background
in both algebraic combinatorics and poset theory.  We encourage
the reader not to skip this section as we make some minor changes
in definitions and standard notation to accommodate partly contradictory
traditions in the two areas.

In the next two sections we present both known and new results in the
order of increasing generality, pointing out the implications between
results along the way.  These implications tend to be quick and
straightforward, and are included for clarity.  In general, we opted
for a complete and detailed presentation of all corollaries and special
cases as a way to fully explain connections between the results.

In a short Section~\ref{s:main}, we present results only about linear extensions
and standard Young tableaux.  While the results are easy consequences of
the $P$-partition results in Section~\ref{s:P-part}, the idea is to make
the linear extension's story completely self-contained.  Our most general
results (Theorems~\ref{t:main-OP-multi} and~\ref{t:main-K})
are given at the end of Section~\ref{s:P-part}.

We then proceed to the proofs. In Section~\ref{s:AD}, we give a self-contained
simple proof of the generalized Fishburn's inequality (Theorem~\ref{t:main-LE}) 
deducing it from its order polynomial generalization (Theorem~\ref{t:main-OP}), 
which is proved via the AD~inequality (Theorem~\ref{thm:AD}).  
This proof is based on Fishburn's approach \cite{Fish},
and is included here as a gentle introduction to our multivariate version.

In Section~\ref{s:AD-multi}, we present the multivariate AD~inequality
(Theorem~\ref{thm:multi-q-AD}).  This is the main tool of the paper,
which we use to prove our main results in a short Section~\ref{s:proof-main}.
In Section~\ref{s:Schur}, we give a new proof of the Lam--Pylyavskyy
inequality for Schur functions, also via the multivariate AD~inequality.

In Section~\ref{s:DDP}, we give a new proof and then a multivariate
generalization (Theorem~\ref{t:DDP-multi}) of the DPP~inequality.
We follow this with the cross-product inequality for $P$-partitions
(Theorem~\ref{thm:CPC-OP}) in Section~\ref{s:CPC}. We conclude
with final remarks and open problems in Section~\ref{s:finrem}.

\medskip

\section{Background, definitions and notation} \label{s:notation}

\subsection{Basic notations}\label{ss:notation-basic}
We use \ts $\nn=\{0,1,2,\ldots\}$, \ts $\nn_{\ge 1} = \{1,2,\ldots\}$,
\ts $[n]=\{1,2,\ldots,n\}$ \ts and \ts $\rr_+ = \{x\ge 0\}$.
To simplify the notation, for an element \ts $a\in X$,  we use
\ts $X-a$ \ts to denote the subset \ts $X\sm \{a\}$.  Similarly,
for a subset \ts $Y\subseteq X$, we write \ts $X-Y$ \ts in place
of more general \ts $X \sm Y$.

For variables \ts $\qb=(q_1,\ldots,q_n)$ \ts and a vector \ts $\bba=(a_1,\ldots,a_n)\in \nn^n$,
we write \ts $\qb^\bba:= q_1^{a_1}\ts \cdots \ts q_n^{a_n}$. For a polynomial \ts
$F\in \rr[z_1,\ldots,z_n]$, we write that \. $F\ge 0$ \. if \.
$F(z_1,\ldots,z_n)\ge 0$ \. for all \ts $\bbz\in \rr^n$.  For two polynomials
$F, G\in \rr[z_1,\ldots,z_n]$, we write \ts $F\ge G$ \ts if \ts $F-G\ge 0$.

For polynomials \ts $F,G\in \rr[z]$, we write \ts $F\geqslant_z G$ \ts if \ts $F-G\in \rr_+[z]$ \ts
is a polynomial with nonnegative coefficients.  For multivariate polynomials
\. $F,G\in \rr[z_1,\ldots,z_n]$, we define \. $F\geqslant_\bbz G$ \. analogously.
We drop the subscript in \ts $\geqslant$ \ts when the variables are clear.
Obviously, \ts $F\geqslant G$ \ts implies \ts $F\geq G$, but not vice versa,
e.g.\ \. $x^2+y^2\geq 2xy$ \. but \. $x^2+y^2\not\geqslant 2xy$.

\smallskip

\subsection{Posets}\label{ss:notation-posets}
We refer to \cite[Ch.~3]{EC} and \cite{Tro} for standard definitions and
notation.
Let \ts $\cP=(X,\prec)$ \ts be a \defnb{partially ordered set}
on the ground set $X$ of size $|X|=n$, and with the partial order~``$\prec$''.
A \defn{subposet} \ts is an induced poset \ts $(Y,\prec)$ \ts on the subset
\ts $Y\subseteq X$. For an element \ts $x \subseteq X$, we denote by \ts $\cP-x$ \ts
the subposet of $\cP$ on \ts $X-x$.

For a poset \ts $\cP=(X,\prec)$, denote by \ts $\cP^\ast = (X, \prec^\ast)$ \ts the
\defn{dual poset} \ts with \ts $x\prec^\ast y$ \ts if and only if \ts $y\prec x$,
for all \ts $x,y \in X$.
For posets \ts $\cP=(X,\prec_\cP)$ \ts and \ts $\cQ=(Y,\prec_\cQ)$,
the \defn{parallel sum} \. $\cP+\cQ=(Z,\prec)$ \. is the poset on the disjoint union \ts $Z=X \sqcup Y$,
where elements of $X$ retain the partial order of~$\cP$, elements of $Y$ retain
the partial order of~$\cQ$, and elements \ts $x \in X$ \ts and  \ts $y \in Y$ are incomparable.
Similarly, the \defnb{linear sum} \. $\cP\oplus \cQ = (Z,\prec)$, where \ts $x\prec y$ \ts for every
two elements \ts $x \in X$ \ts and  \ts $y \in Y$ \ts and other relations as in the parallel sum.

We use \ts $\cC_n$ \ts and \ts $\cA_n$ \ts to denote the $n$-element \defn{chain} \ts
and \defn{antichain}, respectively.  Clearly, \. $\cC_n = \cC_1 \oplus \ts \cdots \ts \oplus \cC_1$ \ts ($n\ts$~times)\.
and \. $\cA_n = \cC_1 + \ts \cdots \ts + \cC_1$ \ts ($n$ \ts times).

A \defn{lattice} \ts is a poset \ts $\LL=(\cL,\prec)$ \ts with meet \ts $x \vee y$ \ts
(least upper bound) and join \ts $x \wedge y$  (greatest lower bound) well defined,
for all \ts $x,y\in \cL$.  We also use \ts $(\cL,\vee,\wedge)$ \ts to denote the lattice
and the join and meet operations.  The lattice \ts $\LL=(\cL,\vee,\wedge)$ \ts
is \defn{distributive} \ts if it satisfied the distributive law:  \. $x\wedge (y \vee z) =
(x\wedge y) \vee (x\wedge z)$. Finally, for all \ts $X,Y \subseteq \cL$, we denote
\[
X \vee Y \, := \, \{ x\vee y \, : \, x\in X, \. y \in Y\} \quad \text{and} \quad
X \wedge Y \, := \, \{ x\wedge y \, : \, x\in X, \. y \in Y\}. \]

%

\smallskip

\subsection{Linear extensions and $P$-partitions}\label{ss:notation-LE}
A \defn{linear extension} of \ts $\cP$ \ts is a bijection \ts $L: X \to [n]$ \ts
that is order-preserving: \ts $x \prec y$ \ts implies \ts $L(x)<L(y)$, for all \ts
$x,y\in X$.
Denote by \ts $\Ec(\cP)$ \ts the set of linear extensions of \ts $\cP$, and let \ts
$e(\cP):=|\Ec(\cP)|$ \ts be the number of linear extensions.  Observe that \ts $e(\cP)=e(\cP^\ast)$ \ts
and \ts $e(\cP\oplus \cQ) = e(\cP)\cdot e(\cQ)$.

A subset \ts $A \subseteq X$ \ts is an \defn{upper ideal} \ts if
\ts $x \in A$ \ts and \ts $y \succ x$ \ts implies \ts $y \in A$.  Similarly,
a subset \ts $A \subseteq X$ \ts is a
\defn{lower ideal} \ts if \. $x \in A$ \ts and \ts $y \prec x$ \ts implies \ts $y \in A$.
We denote by \. $e(A)$ \. the number of linear extensions of the subposet $(A,\prec)$.

Let \ts $\cP=(X,\prec)$, where $X=\{x_1,\ldots,x_n\}$.  We will always assume that \ts $X$
\ts has a \defn{natural labeling}, i.e.\ \ts $L: x_i\to i$ \. is a linear extension.
A \defn{$\cP$-partition} \ts is an order preserving map \. $A: X\to \nn$, i.e.\
maps which satisfy \. $A(x)\le A(y)$ \. for all \. $x \prec y$.
Denote by \ts $\PP(\cP)$ \ts the set of $P$-partitions and let  \ts $\PP(\cP,t)$ \ts be
the set of $P$-partitions with values at most~$t$.\footnote{In \cite{Sta-PP,EC},
Stanley uses $P$-partitions to denote order-reversing rather than order-preserving
maps.  We adopt this version for clarity and to unify the notation.  Displeased readers
can always think of dual posets.  }

Let \. $\Omega(\cP,t):=|\PP(\cP,t)|$ \. be the
number of $\cP$-partitions.  This is the \defn{order polynomial}
corresponding to the poset~$\cP$.\footnote{A standard definition for order polynomial
is  \ts $\Omega(\cP,t-1)$ \ts as the values in the $\cP$-partition are traditionally~$\ge 1$.
We adopt this version to simplify the notation and hope this does not lead to confusion.}
It is well-known and easy to see that
\begin{equation}\label{eq:OP-asy}
\Omega(\cP,t) \ \sim \ \frac{e(\cP)\, t^n}{n!}   \quad \text{as} \ \ t \to \infty\ts,
\quad \text{where \ \. $|X|=n$.}
\end{equation}


Denote \ts $|A|:= \sum_{x\in X} A(x)$ \ts the sum of the entries in a $\cP$-partition.
Let
\begin{equation}\label{eq:OP-q-def}
\Omega_q(\cP,t) \, := \, \sum_{A\in \PP(\cP,t)} \. q^{|A|}\..
\end{equation}
Stanley showed, see \cite[Thm~3.15.7]{EC}, that there is a statistics \. $\maj: \Ec(\cP) \to \nn$,
such that
\begin{equation}\label{eq:Sta-PP}
\Omega_q(\cP,\infty) \ = \ \frac{1}{(1-q)(1-q^2)\cdots (1-q^n)} \, \sum_{A\in \Ec(\cP)} \. q^{\maj(A)}\..
\end{equation}
More generally, let
\begin{equation}\label{eq:OP-bq-def}
\Omega_\qb(\cP,t) \, := \, \sum_{A\in \PP(\cP,t)} \. q_1^{A(x_1)} \ts \cdots \. q_n^{A(x_n)}\..
\end{equation}
We call this GF the \defn{multivariate order polynomial}.  Note that Stanley gave a generalization
of \eqref{eq:Sta-PP} for \ts $\Omega_\qb(\cP,\infty)$ \ts which we will not need,
see \cite[Thm~3.15.5]{EC}. Finally, for $N\geq 0$, define
\begin{equation}\label{eq:OP-rK-def}
\rK_\bbz(\cP,N) \, := \, \sum_{A\in \PP(\cP,N)} \. z_0^{m_0(A)} \ts \cdots \. z_N^{m_N(A)}\.,
\end{equation}
where \ts $m_i(A):= |A^{-1}(i)|$ \ts is the number of values $i$ in the  $\cP$-partition~$A$.

\smallskip

\subsection{Young diagrams and Young tableaux} \label{ss:notation-SYT}
We refer to \cite{Mac,Sag01} and  \cite[Ch.~7]{EC} for standard definitions and notation.
Let \ts $\la=(\la_1,\ldots,\la_\ell)$ \ts be an \defn{integer partition} \ts of~$n$,
write \ts $\la \vdash n$, where \. $\la_1 \ge \la_2 \ge \ldots \ge \la_\ell> 0$ \
and \. $\la_1+\ldots + \la_\ell=n$.  Let \ts $\ell(\la):=\ell$ \ts denotes the number of parts.
A \defn{conjugate partition} \ts $\la'=(\la_1',\la_2',\ldots)$ \ts is defined by
\ts $\la_j' = |\{i\.: \. \la_i\ge j\}|$.

A \defn{Young diagram} \ts is the set of squares \. $\big\{(i,j)\in \nn^2 \. : \. 1\le j \leq  \la_i,
\ts 1\le i \le \ell\big\}$.  In a mild abuse of notation, we use \ts $\la$ \ts to also
denote the corresponding Young diagram, and refer to it as the \defn{straight shape}.
Let \ts $\mu=(\mu_1,\mu_2,\ldots)$ \ts be a partition such that \.
$\mu_i \le \la_i$ \. for all \ts $0\le i \le \ell$.  The difference of Young diagrams
is denoted by \. $\la/\mu$ \. and called the \defn{skew Young diagram} \ts
of shape \ts $\la/\mu$, or simply
the \defn{skew shape}~$\la/\mu$.  We use \ts $|\la/\mu|$ \ts for the size, i.e.\
the number of squares in~$\la/\mu$.

A \defn{standard Young tableau} \ts of shape $\la/\mu$ \ts
is a bijection \ts $A: \la/\mu\to [n]$ \ts which increases in rows and columns: \.
$A(i,j) < A(i+1,j)$ \ts and \ts $A(i,j) <  A(i,j+1)$ \ts whenever these are defined.
Denote by \ts $\SYT(\la/\mu)$ \ts the set of standard Young tableaux of shape~$\la/\mu$.
We note that \ts $|\SYT(\la)|$ \ts can be computed by the
\defng{hook-length formula}, see e.g.\ \cite[$\S$7.21]{EC}.  Similarly,
the number \ts $|\SYT(\la/\mu)|$ \ts  can be computed by the
\defng{Aitken--Feit determinant formula}, see e.g.~\cite[$\S$7.16]{EC}.

Let poset \ts $\cP_{\la/\mu}=(\la/\mu,\prec)$ \ts
be defined by \. $(i,j) \preccurlyeq (i',j')$ \. if \ts $i\leq i'$ \ts
and \ts $j \leq j'$. 
 For example, \ts $P_{31/11} \simeq C_2$ \ts
and \ts $\cP_{321/21} \simeq A_3$.  The set of linear extensions \ts
$\Ec(\cP_{\la/\mu})$ \ts is in bijection with \ts $\SYT(\la/\mu)$,
so \ts $e(\cP_{\la/\mu}) = |\SYT(\la/\mu)|$.

\smallskip

\subsection{Schur functions and reverse plane partitions} \label{ss:notation-Schur}
Let \ts $A: \la/\mu \to \nn$ \ts be a function which increases in rows and columns.
In this context, function \ts $A$ \ts is called a \defn{reverse plane partition}.\footnote{Note
that reverse plane partitions for \ts $\la/\mu$ \ts are actually $\cP_{\la/\mu}$ -- partitions.
This is another notational compromise we make between the areas. }
Let \ts $\RPP(\la/\mu)$ \ts denote the set of reverse plane partition of shape \ts $\la/\mu$.
We think of \ts $A$ \ts as a Young tableau with integers written in squares of~$\la/\mu$.
If \ts $A\in \RPP(\la/\mu)$ \ts is also increasing in columns and has all entries \ts $\ge 1$,
it is called a \defn{semistandard Young tableau}.  The set of such tableaux is denoted
\ts $\SSYT(\la/\mu)$.  We use \ts $\RPP(\la/\mu,t)$ \ts and \ts $\SSYT(\la/\mu,t)$ \ts
to denote reverse plane partitions and semistandard Young tableaux with entries $\le t$.

\defn{Schur polynomial} \ts is a symmetric polynomial associated with the skew shape \ts $\la/\mu$ \ts
and can be defined as
\begin{equation}\label{eq:Schur-def}
s_{\la/\mu}(z_1,\ldots,z_N) \, = \, \sum_{A \ts \in \ts \SSYT(\la/\mu, \ts N)} \. z_1^{m_1(A)} \ts \cdots \. z_N^{m_N(A)}\.,
\end{equation}
where \ts $m_i(A)=|A^{-1}(i)|$ \ts is the number of $i$'s in~$A$.  \defn{Schur functions} \ts are the stable
limits of Schur polynomials as $n\to \infty$.  They form a linear basis in the space of all symmetric
functions.

For reverse plane partitions, observe the connection to the order polynomial:
\begin{equation}\label{eq:Schur-OP-RPP}
\Om(\la/\mu,t) \ := \ \Om(\cP_{\la/\mu}\ts,t) \ = \, \sum_{A \ts \in \ts \RPP(\la/\mu, \ts t)} \. t^{|A|}\..
\end{equation}

In similar manner, consider the following multivariate GF for the reverse plane partitions:
$$
\RGF_{\la/\mu}(z_0,z_1,\ldots,z_N) \, = \, \sum_{A \ts \in \ts  \RPP(\la/\mu, \ts N)} \. z_0^{m_0(A)} \. z_1^{m_1(A)} \. \cdots \. z_N^{m_N(A)}\.,
$$
Note the notation above, we have \. $\RGF_{\la/\mu}(z_0,z_1,\ldots, z_N) = \rK_\bbz\big(\cP_{\la/\mu},N\big)$.


\medskip

\section{Linear extensions} \label{s:main}
%
%
%

\subsection{Fishburn's inequality}\label{ss:main-Fish}
We start with the following fundamental inequality:

\smallskip

\begin{thm}[{\rm \defn{Fishburn's inequality} \cite{Fish1}}{}]
\label{t:Fish}
Let \ts $\cP=(X,\prec)$ \ts be a finite poset, and let \ts $A,B \ssu X$ \ts be lower ideals of~$P$.
Then:
\begin{equation}\label{eq:Fish}
	  \frac{e(A \cup B) \.\cdot\. e(A\cap B)}{e(A) \.\cdot\. e(B)}  \
\geq \ \frac{|A \cup B|! \. \cdot \. |A \cap B|!}{|A|! \. \cdot \.  |B|!}\,.
\end{equation}
\end{thm}
\smallskip

Using the notation
$$
\en(P) \. := \. \frac{e(P)}{|X|!} \.,
$$
Fishburn's inequality can be rewritten in a more concise form
as a correlation inequality for probabilities:
\begin{equation}\label{eq:Fish-prob}
	  \en(A \cup B) \.\cdot\. \en(A\cap B) \
\geq \en(A) \.\cdot\. \en(B)\,.
\end{equation}
The original proof of Fishburn's inequality uses the \defng{AD inequality}.  Note that it is tight
for the antichain \ts $\cP=\cA_n$.

\smallskip

\subsection{Bj\"{o}rner's inequality}\label{ss:main-Bjo}
For a skew Young diagram \ts $|\la/\mu|=n$, we similarly denote
$$
\fe(\la/\mu) \, := \, f(\cP_{\la/\mu}) \, = \, \frac{|\SYT(\la/\mu)|}{n!}\..
$$
Now \eqref{eq:Fish-prob} gives:

\smallskip

\begin{cor}[{\rm \defn{Bj\"{o}rner's inequality} \cite[$\S$6]{Bjo11}}{}]\label{c:Bjo}
Let \ts $\mu$ \ts and \ts $\nu$ \ts be Young diagrams.  Then:
	\begin{equation}\label{eq:Bjo}
	\fe(\mu\vee \nu) \. \cdot \.  \fe(\mu\wedge \nu) \ \ge \ \fe(\mu) \. \cdot \. \fe(\nu)\ts,
	\end{equation}
	where \ts $\vee$ \ts and \ts $\wedge$ \ts refer to the union and intersection of the Young diagrams.
\end{cor}

\smallskip

Bj\"{o}rner's proof used another Fishburn's result combined with the some calculations
using the hook-length formula.  The following result has an ambiguous status of being
nominally new, yet it easily follows from the LP inequality (see $\S$\ref{ss:P-part-LP} below).

\smallskip

\begin{cor}[{\rm \defn{generalized Bj\"{o}rner's inequality}}{}]\label{c:Bjo-gen}
Let \ts $\mu/\al$ \ts and \ts $\nu/\be$ \ts be skew Young diagrams.  Then:
	\begin{equation}\label{eq:Bjo-gen}
	\fe(\mu/\al \ts \vee \ts \nu/\be) \. \cdot \.  \fe(\mu/\al \ts \wedge \ts \nu/\be) \ \ge \ \fe(\mu/\al) \. \cdot \. \fe(\nu/\be)\ts.
	\end{equation}
\vskip.18cm
\nin
	where \, $\mu/\al \. \vee \. \nu/\be \ts := \ts (\mu\vee \nu)/(\al\vee \be)$ \, and \,
$\mu/\al \. \wedge \. \nu/\be \ts := \ts (\mu\wedge \nu)/(\al\wedge \be)$.
\end{cor}

\smallskip

In contrast with Bj\"{o}rner's inequality, the generalized Bj\"{o}rner inequality does not
follow from Fishburn's inequality, at least not directly.

\smallskip

\subsection{Generalized Fishburn's inequality}\label{ss:main-dual}
Our first new result is a common generalization of both the Fishburn's and the
generalized Bj\"orner's inequalities.

\smallskip

\begin{thm}\label{t:main-LE}
Let \ts $\cP=(X,\prec)$ \ts be a finite poset. Let \ts $A,B \subseteq X$ \ts be lower ideals,
and let \ts $C,D \subseteq X$ \ts be upper ideals of~$\cP$, such that \.
$A \cap C = B \cap D = \varnothing$.
Then:
\begin{equation}\label{eq:main-LE}
\en(X-V) \.\cdot\. \en(X-W) \ \geq \ \en(X-A-C) \.\cdot\. \en(X-B-D)\,.
\end{equation}
where \. $V:=(A \cap B) \cup (C \cup D)$ \. and \. $W:=(A \cup B) \cup (C \cap D)$.
\end{thm}

\smallskip

Note that Fishburn's inequality (Theorem~\ref{t:Fish}) is a special case  \ts $C=D=\emp$, and that
Theorem~\ref{t:main-LE} is self-dual.  We prove the theorem using the AD inequality in Section~\ref{s:AD}.

\smallskip

\begin{proof}[Proof of  \, {\rm $[$Theorem~\ref{t:main-LE} \, $\Longrightarrow$ \, Corollary~\ref{c:Bjo-gen}$]$}]
Let \ts $\cP:= \cP_\la$, where \ts $\la:= \mu \vee \nu$.  In the notation of Theorem~\ref{t:main-LE}, we have \ts $X=\la$.
Consider the following four subsets of the Young diagram \ts $\la$:
\begin{equation}\label{eq:subs-Young}
A \. := \. \al\., \ \ B \. := \. \be\., \ \ C \. := \. \la/\mu\., \ \ D \. := \. \la/\nu\..
\end{equation}
Now observe that
$$X-A-C \. = \.  \mu/\alpha \ts, \ \  X-B-D \. = \.  \nu/\beta\., \ \  X -V \. = \.  (\mu \wedge \nu)/(\alpha \wedge  \beta)\ts,
 \ \  X -W \. = \.  (\mu \vee \nu)/(\alpha \vee \beta)
\ts.
$$
Thus, \eqref{eq:main-LE} implies \eqref{eq:Bjo-gen}, as desired.
\end{proof}

\medskip


\section{$P$-partitions}\label{s:P-part}

\subsection{Schur functions}\label{ss:P-part-LP}
The following \defn{LP inequality} \ts is the key result which inspired
this paper.

\smallskip

\begin{thm}[{\rm \defn{Lam--Pylyavskyy inequality for Schur polynomials} \cite[Thm~4.5]{LP07}}{}]
\label{t:P-Schur}
Let \ts $\mu/\al$ \ts and \ts $\nu/\be$ \ts be skew Young diagrams, and let \ts
$\bbz=(z_1,\ldots,z_N)$, where \ts $N\ge \ell(\mu),\ell(\nu)$.
Then:
	\begin{equation}\label{eq:LP1}
	s_{\mu\vee \nu}(\bbz) \. \cdot \.  s_{\mu\wedge \nu}(\bbz) \, \
\ges_\bbz \, \ s_{\mu}(\bbz) \. \cdot \. s_{\nu}(\bbz)\,.
	\end{equation}
More generally, we have:
	\begin{equation}\label{eq:LP2}
	s_{\mu/\al \. \vee \. \nu/\be}(\bbz) \. \cdot \.  s_{\mu/\al \. \wedge \. \nu/\be}(\bbz) \, \
\ges_\bbz \, \ s_{\mu/\al}(\bbz) \. \cdot \. s_{\nu/\be}(\bbz)\.,
	\end{equation}

\vskip.18cm
\nin
	where \. $\mu/\al \. \vee \. \nu/\be \ts := \ts (\mu\vee \nu)/(\al\vee \be)$ \. and \.
$\mu/\al \. \wedge \. \nu/\be \ts := \ts (\mu\wedge \nu)/(\al\wedge \be)$.
\end{thm}

\smallskip

The original proof is completely combinatorial and uses an explicit injection.
For completeness, we include a short argument showing how the LP inequality implies
the Bj\"orner's and the generalized Bj\"orner's inequality.

\smallskip

\begin{proof}[Proof of \. \eqref{eq:LP2} \. $\Longrightarrow$ \. \eqref{eq:Bjo-gen}]
Recall the following analogue of \eqref{eq:Sta-PP} for skew Schur functions:
\begin{equation}\label{eq:P-part}
s_{\la/\tau}(1,q,q^2,\ldots) \ = \ \frac{1}{(1-q)(1-q^2)\cdots (1-q^{|\la/\tau|})} \,\. \sum_{T \ts \in \ts \SYT(\la/\tau)} \. q^{\maj(T)}\.,
\end{equation}
where \ts $\maj: \SYT(\la/\tau) \to \nn$ \ts is the \emph{major index} \ts of a tableau, see e.g.\ \cite[Thm~7.19.11]{EC}.

Let \ts $n:=|\mu/\al|+|\nu/\be|$. Substituting \eqref{eq:P-part} into each of the four Schur
functions in the LP inequality \eqref{eq:LP2}, multiplying both sides by \. $(1-q)^{n}$ \.
and letting \ts $q\to 1$, gives the generalized Bj\"orner's inequality \eqref{eq:Bjo}.
\end{proof}

\begin{rem}  \label{r:LPP}
The following truly remarkable \defn{Lam--Postnikov--Pylyavskyy inequality} \ts
further extended \eqref{eq:LP2} and resolved several open problems in the area:
\begin{equation}\label{eq:LPP2}
	s_{\mu/\al \ts \vee \ts \nu/\be} \. \cdot \.  s_{\mu/\al \ts \wedge \ts \nu/\be} \, \
\ges_s \, \ s_{\mu/\al} \. \cdot \. s_{\nu/\be}\,.
\end{equation}
Here ``$\geqslant_s$'' stands for \defn{Schur positivity}, which is saying that
the difference is a nonnegative sum of Schur functions. Although we will not need
this extension, it does give a more conceptual proof of Bj\"orner's inequality.

In a different direction, Richards~\cite{Rich} gave an analytic generalization of \eqref{eq:LP1}
for real \ts $\la,\mu\in \rr^\ell$ \ts and the determinant definition of Schur
polynomials.  It would be natural to conjecture that \eqref{eq:LPP2} also generalizes
to this setting.
\end{rem}

\smallskip

\begin{proof}[Proof of \. \eqref{eq:LPP2} \. $\Longrightarrow$ \. \eqref{eq:Bjo}]
Recall that for all \ts $\mu \vdash k$, \ts $\nu \vdash n-k$, we have:
$$
s_\mu \cdot s_\nu \, = \, \sum_{\la\vdash n} \. c^{\la}_{\mu\nu} \ts s_{\la}
\qquad \text{and} \qquad
\chi^\mu \otimes \chi^\nu \uparrow_{S_k \times S_{n-k}}^{S_n} \ = \ \sum_{\la\vdash n} \. c^{\la}_{\mu\nu} \. \chi^\la\.,
$$
where \ts $c^{\la}_{\mu\nu}$ \ts are the \defn{Littlewood--Richardson coefficients},
see e.g.~\cite[$\S$4.9]{Sag01}. Equating dimensions in the second equality gives:
$$
f(\mu) \cdot f(\nu) \, = \,
\sum_{\la\vdash n} \. c^{\la}_{\mu\nu} \. f(\la)\ts.
$$
Thus \. $\vp: s_\la \to f(\la)$ \. is a ring homomorphism from the ring
of symmetric function to~$\qqq$ which maps Schur positive symmetric function
to~$\qqq_+$\ts.  Applying \ts $\vp$ \ts to the inequality \eqref{eq:LPP2}
for \ts $\al=\be=\emp$ \ts gives the desired inequality \eqref{eq:Bjo}.
\end{proof}

\smallskip

\subsection{RPP variation}\label{ss:P-part-LP-RPP-var}
The following RPP variation is an easy corollary of the LP inequality \eqref{eq:LP2}:

\smallskip

\begin{cor}\label{c:LP-RPP1}
Let \ts $\mu$ \ts and \ts $\nu$ \ts be Young diagrams and let \ts $t\ge 0$.  Then:
	\begin{equation}\label{eq:LP-RPP0}
\Om(\mu\vee \nu,t) \. \cdot \.  \Om(\mu\wedge \nu,t)  \
\geq  \ \Om(\mu,t) \. \cdot \.  \Om(\nu,t).
	\end{equation}
Similarly, for the $q$-statistics we have:
	\begin{equation}\label{eq:LP-RPP-q}
\Om_q(\mu\vee \nu,\infty) \. \cdot \.  \Om_q(\mu\wedge \nu,\infty)  \
\geqslant_q  \ \Om_q(\mu,\infty) \. \cdot \.  \Om_q(\nu,\infty).
	\end{equation}
More generally, we have:
	\begin{equation}\label{eq:LP-RPP1}
\Om_q(\mu\vee \nu,t) \. \cdot \.  \Om_q(\mu\wedge \nu,t)  \
\geqslant_q  \ \Om_q(\mu,t) \. \cdot \.  \Om_q(\nu,t).
	\end{equation}
\end{cor}

\smallskip

\begin{proof}[Proof of \. \eqref{eq:LP1} \. $\Longrightarrow$ \. \eqref{eq:LP-RPP-q}]
Setting \. $N\gets \infty$ \. and \. $\bbz = (z_1,z_2,\ldots) \gets (q,q,\ldots)$, we get:
\begin{equation}\label{eq:LP-RPP-Schur}
s_\la(q,q,\ldots) \, = \, \Om_q(\la,\infty) \cdot q^{n(\la)}\., \quad \text{where} \ \ \. n(\la) \. = \.
\sum_{(i,j)\in \la} \ts i\ts.
\end{equation}
Note that \. $n(\mu\vee \nu) + n(\mu\wedge \nu) = n(\mu) + n(\nu)$.  Substituting \eqref{eq:LP-RPP-Schur}
into \eqref{eq:LP1} and dividing both sides by \. $q^{n(\mu) + n(\nu)}$ \. gives \eqref{eq:LP-RPP-q}.
\end{proof}

\smallskip

\begin{cor}\label{c:LP-OP}
Let \ts $\cP=(X,\prec)$ \ts be a finite poset, let \ts $t\ge 0$,
and let \ts $A,B \ssu X$ \ts be lower ideals of~$P$.
Then:
\begin{equation}\label{eq:LP-OP1}
\Om(A\cup B,t) \. \cdot \. \Om(A\cap B,t) \ \geq \
\Om(A,t) \. \cdot \. \Om(B,t)\ts.
\end{equation}
More generally, we have:
\begin{equation}\label{eq:LP-OP2}
\Om_q(A\cup B,t) \. \cdot \. \Om_q(A\cap B,t) \ \geqslant_q \
\Om_q(A,t) \. \cdot \. \Om_q(B,t)\ts.
\end{equation}
\end{cor}

\smallskip

\begin{proof}[Proof of \. \eqref{eq:LP-OP1} \. $\Longrightarrow$ \. \eqref{eq:Fish-prob}]
Let \ts $t\to \infty$ \ts and apply \eqref{eq:OP-asy} to each term in \eqref{eq:LP-OP1}.
\end{proof}

\smallskip

Corollary~\ref{c:LP-OP} is a direct generalization of Corollary~\ref{c:LP-RPP1},
which follows by  taking \ts $\A\gets \mu$ \ts and \ts $B\gets \nu$.
Our next result is a multivariate generalization of Corollary~\ref{c:LP-RPP1}.

\smallskip

\begin{cor}\label{c:LP-RPP2}
Let \ts $\mu$ \ts and \ts $\nu$ \ts be Young diagrams and let \ts $N\ge 0$.  Then:
	\begin{equation}\label{eq:LP-RPP2}
	\RGF_{\mu\vee \nu}(\bbz) \. \cdot \.  \RGF_{\mu\wedge \nu}(\bbz)  \
\geqslant_{\bbz} \ \RGF_{\mu}(\bbz) \. \cdot \. \RGF_{\nu}(\bbz)\,,
	\end{equation}
where \. $\bbz = (z_0,z_1,\ldots,z_N)$.
\end{cor}

\smallskip

\begin{proof}[Proof of \. \eqref{eq:LP-RPP2} \. $\Longrightarrow$ \. \eqref{eq:LP-RPP1}]
Let \ts $N\gets t$, and set \ts $z_i \gets q^i$ \ts for all \ts $0\le i \le N$.
\end{proof}

\smallskip

This result is implicit in \cite{LP07} and follows from the following general theorem:

\smallskip

\begin{thm}[{\rm \defn{Lam--Pylyavskyy inequality for multivariate order polynomials} \cite[Prop.~3.7]{LP07}}{}] \label{t:LP-poset}
Let \ts $\cP=(X,\prec)$ \ts be a finite poset,  let \ts $A,B \ssu X$ \ts be lower ideals of~$P$, and let $N\geq 0$.
Then:
\begin{equation}\label{eq:LP-RPP3}
\rK_\bbz(A\cup B,N) \. \cdot \. \rK_\bbz(A\cap B,N) \ \geqslant_\bbz \
\rK_\bbz(A,N) \. \cdot \. \rK_\bbz(B,N)\ts.
\end{equation}
\end{thm}

\smallskip

This is the most general version of the LP~inequality that we discuss in this
paper.  Note that \. \eqref{eq:LP-RPP3} \ts $\Longrightarrow$ \ts \eqref{eq:LP-RPP2} \. by taking
\ts $\A\gets \mu$ \ts and \ts $B\gets \nu$.

\smallskip

\begin{rem}\label{r:LP-RPP}
As we mention in the introduction, the ultimate Lam--Pylyavskyy generalization
uses the meet and join operations which are incompatible with those we employ in
this paper.  They are in fact, noncommutative and designed to allow the
``cell transfer'' direct injection.

Notably, \eqref{eq:LP2} does not follow from \eqref{eq:LP-RPP3}, but from
the proof of this ultimate Lam--Pylyavskyy generalization which happens to
apply to skew shapes.  We give a more streamlined derivation of \eqref{eq:LP2}
from our generalization below.
\end{rem}


\smallskip

\subsection{Main results}\label{ss:P-part-main}
We begin with the order polynomial extension of the generalized Fishburn's inequality
(Theorem~\ref{t:main-LE}) and the Lam--Pylyavskyy order polynomial inequality (Corollary~\ref{c:LP-OP}).

\smallskip

\begin{thm}\label{t:main-OP}
Let \ts $\cP=(X,\prec)$ \ts be a finite poset. Let \ts $A,B \subseteq X$ \ts be lower ideals,
and let \ts $C,D \subseteq X$ \ts be upper ideals of~$\cP$, such that \.
$A \cap C = B \cap D = \varnothing$.
Then:
\begin{equation}\label{eq:main-OP1}
\Om(X-V,t) \.\cdot\. \Om(X-W,t) \ \geq \ \Om(X-A-C,t) \.\cdot\. \Om(X-B-D,t)\ts,
\end{equation}
where \. $V:=(A \cap B) \cup (C \cup D)$ \. and \. $W:=(A \cup B) \cup (C \cap D)$.
More generally, we have:
\begin{equation}\label{eq:main-OP2}
\Om_q(X-V,t) \.\cdot\. \Om_q(X-W,t) \ \geqslant_q \ \Om_q(X-A-C,t) \.\cdot\. \Om_q(X-B-D,t)\ts.
\end{equation}
\end{thm}

\smallskip

Corollary~\ref{c:LP-OP} is a special case of the theorem when \ts $C=D= \emp$.

\smallskip

\begin{proof}[Proof of \. \eqref{eq:main-OP1} \. $\Longrightarrow$ \. \eqref{eq:main-LE}]
Let \ts $t\to \infty$ \ts and apply \eqref{eq:OP-asy} to each term in \eqref{eq:main-OP1}.
\end{proof}

\smallskip

Here is our most general result in this direction, and the ultimate multivariate
generalization of Fishburn's inequality (Theorem~\ref{t:Fish}).

\smallskip

\begin{thm}\label{t:main-OP-multi}
Let \ts $\cP=(X,\prec)$ \ts be a finite poset. Let \ts $A,B \subseteq X$ \ts be lower ideals,
and let \ts $C,D \subseteq X$ \ts be upper ideals of~$\cP$, such that \.
$A \cap C = B \cap D = \varnothing$.
Then:
\begin{equation}\label{eq:main-OP3}
\Om_\qb(X-V,t) \.\cdot\. \Om_\qb(X-W,t) \ \geqslant_\qb \ \Om_\qb(X-A-C,t) \.\cdot\. \Om_\qb(X-B-D,t)\ts,
\end{equation}
where \. $V:=(A \cap B) \cup (C \cup D)$ \. and \. $W:=(A \cup B) \cup (C \cap D)$.
\end{thm}

\smallskip

\begin{proof}[Proof of \. \eqref{eq:main-OP3} \. $\Longrightarrow$ \. \eqref{eq:main-OP2}]
Take \ts $\qb \gets (q,\ldots,q)$.
\end{proof}

\smallskip

Finally, we present another generalization of Theorem~\ref{t:main-OP}
for different choices of rank functions, and furthermore generalizes
Lam--Pylyavskyy Theorem~\ref{t:LP-poset}.  We prove both theorems
in Section~\ref{s:proof-main}.

\smallskip

\begin{thm}\label{t:main-K}
	Let \ts $\cP=(X,\prec)$ \ts be a finite poset. Let \ts $A,B \subseteq X$ \ts be lower ideals,
	and let \ts $C,D \subseteq X$ \ts be upper ideals of~$\cP$, such that \.
	$A \cap C = B \cap D = \varnothing$.
	Then:
	\begin{equation}\label{eq:main-OP-K}
		\rK_\bbz(X-V,N) \.\cdot\. \rK_\bbz(X-W,N) \ \geqslant_\bbz \ \rK_\bbz(X-A-C,N) \.\cdot\. \rK_\bbz(X-B-D,N)\ts,
	\end{equation}
	where \. $V:=(A \cap B) \cup (C \cup D)$ \. and \. $W:=(A \cup B) \cup (C \cap D)$.
\end{thm}

\smallskip

\begin{proof}[Proof of \. \eqref{eq:main-OP-K} \. $\Longrightarrow$ \. \eqref{eq:main-OP2}]
Take \ts $\bbz \gets (1,q,q^2,\ldots,q^N)$.
\end{proof}

\smallskip

In particular, these two theorems imply the following corollary for skew Young diagrams.

\smallskip

\begin{cor}\label{c:RPP-two}
	Let \ts $\mu/\al$ \ts and \ts $\nu/\be$ \ts be skew Young diagrams.  Then:
\begin{align}\label{eq:main-RPP-Om}
		\Om_{\qb}(\mu/\al \. \vee \. \nu/\beta,t) \. \cdot \.  \Om_{\qb}(\mu/\al \. \wedge \. \nu/\be,t)  \
	&\geqslant_{\qb} \ \Om_{\qb}(\mu/\al,t) \. \cdot \. \Om_{\qb}(\nu/\be,t)\,,
\end{align}
and
\begin{align}	\label{eq:main-RPP-F}
	\RGF_{\mu/\al \. \vee \. \nu/\beta}(\bbz) \. \cdot \.  \RGF_{\mu/\al \. \wedge \. \nu/\be}(\bbz)  \
	&\geqslant_{\bbz} \ \RGF_{\mu/\al}(\bbz) \. \cdot \. \RGF_{\nu/\be}(\bbz)\,,
\end{align}
where \. $\bbz = (z_0,z_1,\ldots,z_N)$.
\end{cor}


\begin{proof}
	Let \. $\cP,A,B,C,D$ \. be as in \eqref{eq:subs-Young}. By applying the same argument as in
the proof of the \. $[$Theorem~\ref{t:main-LE} \, $\Longrightarrow$ \, Corollary~\ref{c:Bjo-gen}$]$ \.
implication, the inequality \eqref{eq:main-RPP-Om} now follows from \eqref{eq:main-OP3},
 while the inequality \eqref{eq:main-RPP-F} follows from \eqref{eq:main-OP-K}.
\end{proof}

\smallskip

\begin{rem}\label{r:RPP-LP}
Although the inequalities \eqref{eq:main-RPP-F} and \eqref{eq:main-RPP-Om}
do not appear in \cite{LP07}, they follow from the approach in that paper.
\end{rem}

\medskip

%
%

\section{The Ahlswede–-Daykin inequality}\label{s:AD}

In this section, we prove the first part of Theorem~\ref{t:main-OP} by using the
Ahlswede--Daykin (AD) inequality.  Our approach is based on the proof
in \cite{Fish1}.
For every \. $\rho: Z\to \Rb_{+}$ \. and every \. $X\subseteq Z$, denote
\begin{equation}\label{eq:def-rho-sum}
 \rho(X) \ := \ \sum_{x \in X} \. \rho(x)\ts.
\end{equation}

\smallskip

\begin{thm}[{\rm \defn{Ahlswede--Daykin inequality}~\cite{AD}}{}]\label{thm:AD}
	Let \ts $\LL=(\cL,\vee,\wedge)$ \ts be a  finite distributive lattice, and let \. $\alpha,\beta,\gamma,\delta: \cL \to \Rb_{+}$ \. be nonnegative functions on $\cL$.
	Suppose we have:
\begin{equation}\label{eq:AD-def}
 \alpha(x) \.\cdot \. \beta(y) \ \leq \ \gamma(x \vee y) \.\cdot \. \delta(x \wedge y) \quad \text{ for every } \ \ x, \ts y \in \cL\ts.
 \end{equation}
	Then:
\begin{equation}\label{eq:AD}
 \alpha(X) \.\cdot \. \beta(Y) \ \leq \ \gamma(X \vee Y) \.\cdot \. \delta(X \wedge Y) \quad \text{ for every } \ \ X, \ts Y \subseteq \cL\ts.
\end{equation}
\end{thm}

\smallskip

\subsection*{Proof of the first part of Theorem~\ref{t:main-OP}}
Let \ts $\cP=(X,\prec)$ \ts be a poset, and let $t\ge 0$.
We denote by \. $\LL(\cP,t)=(\cL,\vee,\wedge)$ \. the distributive lattice on the set \. $\cL \subseteq \{0,\ldots,t\}^{X}$ \. given by
\begin{equation}\label{eq:lat-Pt}
 \cL \ := \ \PP(\cP,t) \ = \ \big\{  T:X \to \{0,\ldots,t\} \ : \  T(x) \leq T(y) \ \ \text{ for all } \ \. x,y \in X \. \ \text{s.t.} \ \. x \prec y \big\},
\end{equation}
with the join and meet operation given by
\begin{align*}
	[S \vee T](x) \ = \ \max\{S(x), T(x)\} \quad \text{and}
\quad [S \wedge  T](x) \ = \    \min\{S(x), T(x)\} \quad \text{for every} \ x \in X.
\end{align*}
Recall that \. $\Om(\cP,t) = |\cL|$.
%
%
%
%
%
	Let \. $\alpha,\beta,\gamma, \delta: \cL \to \Rb_{+}$ \. be given by
	\begin{equation}\label{eq:Fish-ftn}
	\begin{split}
		\alpha(T) \ &:= \ \mathbf{1}\{ \. T(x) =0 \ \text{ for all } \  x \in A\., \ T(y)= t \ \text{ for all } \ y \in C \},\\
		\beta(T) \ &:= \ \mathbf{1}\{ \. T(x) =0 \ \text{ for all } \  x \in B\., \ T(y)= t \ \text{ for all } \ y \in D \},\\
		\gamma(T) \ &:= \ \mathbf{1}\{ \. T(x) =0 \ \text{ for all } \  x \in A \cap B\., \ T(y)= t \ \text{ for all } \ y \in C \cup D \},\\
		\delta(T) \ &:= \ \mathbf{1}\{ \. T(x) =0 \ \text{ for all } \  x \in A \cup B\., \ T(y)= t \ \text{ for all } \ y \in C \cap D \}.
	\end{split}
	\end{equation}
Note that
\begin{alignat*}{2}
	\alpha(\cL) \ &= \  \Omega(X-A-C,t), \qquad && \beta(\cL)  \ = \ \Omega(X-B-D,t),\\
	\gamma(\cL) \ &= \  \Omega(X-V,t), \qquad   &&
	\delta(\cL) \ = \  \Omega(X-W,t).
\end{alignat*}
By the AD inequality~\eqref{eq:AD}, it thus suffices to verify \eqref{eq:AD-def}, which in this case
states:
\begin{equation}\label{eq:Fish-AD-cond}
 \alpha(S) \.\cdot \. \beta(T) \ \leq \ \gamma(S \vee T) \.\cdot \. \delta(S \wedge T) \quad \text{ for every } \ S,T \in \cL.
\end{equation}
Let \ts $S,T \in \cL$ \ts be such that 	 \. $\alpha(S)=\beta(T)=1$.
Then:
\begin{align*}
 S(x) =0 \ \text{ for } x \in A,   \quad S(y) =t \ \text{ for } y \in C, \quad T(x) =0 \ \text{ for } x \in B,   \quad T(y) =t \ \text{ for } y \in D.
\end{align*}
This gives:
\begin{align*}
	& \max \{S(x), T(x)\} \ = \ 0 \ \text{ for } \ x \in A \cap B, \quad
	& \max \{S(y), T(y)\} \  = \ t \ \text{ for } \ x \in C \cup D, \\
		& \min \{S(x), T(x)\} \ = \ 0 \ \ts \text{ for } \ x \in A \cup B, \quad
	& \min \{S(y), T(y)\} \ = \ t \ \text{ for } \ x \in C \cap D.
\end{align*}
The first equation implies \. $\gamma(S \vee T)=1$\., while the second equation
implies \. $\delta(S\wedge T)=1$. This implies \eqref{eq:Fish-AD-cond} and
completes the proof of \eqref{eq:main-OP1}.
\qed

\smallskip

\begin{rem} \label{r:AD-q}
For the second (more general) part of Theorem~\ref{t:main-OP}, one can use the
same approach with the  AD~inequality in Theorem~\ref{thm:AD} replaced with
$q$-AD~inequality by Christofides~\cite{Chri}.  Our proof of
Theorem~\ref{t:main-OP-multi} given below, extends Theorem~\ref{t:main-OP}
using the multivariate $\qb$-AD~inequality.
\end{rem}

\bigskip


\medskip

\section{Multivariate AD inequality}\label{s:AD-multi}

\subsection{The statement} \label{ss:AD-multi-thm}
Let \. $\LL:=(\cL,\wedge,\vee)$ \. be a finite distributive lattice.
Throughout this section, fix variables \ts $q_1,\ldots,q_\ell$\ts, and \defn{modular functions}
\.  $r_1,\ldots, r_\ell:\cL \to \nn$ \. defined to satisfy
\[ r_i(x)+r_i(y) \ = \  r_i(x \vee y) + r_i(x \wedge  y) \quad \text{ for all } \ \ x,y \in \cL \ \ \text{and} \ \, \ 1\le i \le \ell. \]
Write \. $\qb:=(q_1,\ldots,q_\ell)$ \. and \. $\rb:=(r_1,\ldots, r_\ell)$.
For $x \in \cL$, write
\[ \rb(x) \ := \ \big(r_1(x), \ldots, r_\ell(x)\big)  \quad \text{and} \quad \qb^{\rb(x)} \ := \ q_1^{r_1(x)} \cdots \. q_\ell^{r_\ell(x)}.  \]
For a function \. $\rho: \cL \to \Rb_{+}$ \. and subset \. $X \subseteq \cL$, define
\begin{equation}\label{eq:def-rho-sum-q}
\rho_{\langle\qb,\rb \rangle}(X) \ := \  \sum_{x\in X} \. \rho(x) \. \qb^{\rb(x)} \ \in  \Rb_+[q_1,\ldots,q_\ell].
\end{equation}
Note that \eqref{eq:def-rho-sum-q} is a multivariate $\qb$-analogue of \eqref{eq:def-rho-sum}.
We can now state the multivariate $\qb$-analogue of the Ahlswede--Daykin inequality (Theorem~\ref{thm:AD}).

%
%
%
%
%

\smallskip

\begin{thm}[{\rm \defn{multivariate AD inequality}}{}]\label{thm:multi-q-AD}
Let \. $\LL=(\cL,\wedge,\vee)$ \ts be a  finite distributive lattice, and let \. $\alpha,\beta,\gamma,\delta: \cL \to \Rb_{+}$ \. be nonnegative functions on $\cL$.
Suppose we have
\begin{equation}\label{eq:AD-condition}
 \alpha(x) \.\cdot \. \beta(y) \ \leq \ \gamma(x \vee y) \. \cdot \.\delta(x \wedge y) \quad \text{ for every } \ \ x,y \in \cL.
\end{equation}
Then:
\begin{equation}\label{eq:AD-multi}
 \alpha_{\lqr}(X) \. \cdot \.\beta_{\lqr}(Y) \ \leqs_{\qb} \ \gamma_{\lqr}(X \vee Y) \.\cdot \. \delta_{\lqr}(X \wedge Y) \quad \text{ for every } \ \ X,Y \subseteq \cL.
\end{equation}
\end{thm}

\smallskip

Our proof is strongly inspired by those of Bj\"orner~\cite{Bjo11} and Christofides~\cite{Chri}.
We closely follow the presentation from the former while incorporating some ideas from the
latter paper.

\subsection{The proof} \label{ss:AD-multi-proof}
We start by proving the following special case of Theorem~\ref{thm:multi-q-AD},
which we use to obtain the theorem in the full generality.

\smallskip

\begin{prop}\label{prop:multi-q-AD-spec}
	Let \ts $\LL=(\cL,\wedge,\vee)$, $\alpha$, $\beta$, $\gamma$, $\delta$ \ts be as in Theorem~\ref{thm:multi-q-AD}.
	Then:
\begin{equation}\label{eq:prop-AD-multi}
\alpha_{\lqr}(\cL) \. \beta_{\lqr}(\cL) \ \leqs_{\qb} \ \gamma_{\lqr}(\cL) \. \delta_{\lqr}(\cL).
\end{equation}	
\end{prop}

\smallskip

\begin{proof}[Proof of \, Proposition~\ref{prop:multi-q-AD-spec} \ $\Longrightarrow$ \ts  Theorem~\ref{thm:multi-q-AD}]
	Let \. $\alpha',\beta',\gamma',\delta':\cL \to \Rb_+$ \. be functions given by
	\[ \alpha' \ := \  \alpha \circ \mathbf{1}_{X}\,, \qquad \beta' \ := \  \beta \circ \mathbf{1}_{Y}\,, \quad \gamma' \ := \  \gamma \circ \mathbf{1}_{X \vee Y}\,, \quad \delta \ := \  \delta' \circ \mathbf{1}_{X \wedge Y}\,.  \]
	Note that
\begin{equation}\label{eq:four-functions-prime}
	\alpha'(x) \.\cdot \. \beta'(y) \ \leq \ \gamma'(x \vee y) \.\cdot \. \delta'(x \wedge y) \quad \text{ for every } \ x,y \in \cL.
\end{equation}
	Indeed, the LHS of \eqref{eq:four-functions-prime} is equal to $0$ if  \. $x \notin A$ \. or \. $y \notin B$\.,
	so suppose that \. $x \in A, \. y \in B$\..
	Then the inequality reduces to \eqref{eq:AD-condition}, which is part of the assumption.
	The inequality \eqref{eq:AD-multi} then follows from \eqref{eq:prop-AD-multi} by noting that
\begin{alignat*}{2}
 \alpha'_{\lqr}(\cL) \  &= \ \alpha_{\lqr}(X), \qquad &&\beta'_{\lqr}(\cL) \  = \ \beta_{\lqr}(Y), \\
  \gamma'_{\lqr}(\cL) \  &= \ \gamma_{\lqr}(X \vee Y), \qquad &&\delta'_{\lqr}(\cL) \  = \ \delta_{\lqr}(X \wedge Y),
\end{alignat*}
	as desired.
\end{proof}

\smallskip

\begin{proof}[Proof of Proposition~\ref{prop:multi-q-AD-spec}]
	Let
\[ \Phi(\qb,\rb) \ := \  \alpha_{\lqr}(\cL)  \.\cdot \.  \beta_{\lqr}(\cL) \ - \ \gamma_{\lqr}(\cL)  \.\cdot \.  \delta_{\lqr}(\cL).\]
For $x,y \in \cL$, we also define
\begin{align*}
	\phi(x,y) \ &:= \  \alpha(x) \.\cdot \. \beta(y) \, - \, \gamma(x) \.\cdot \.  \delta(y).
\end{align*}
A simple computation shows that
\[ \Phi(\qb,\rb) \ = \ \sum_{(x,y) \ts \in \ts \cL^2}   \phi(x,y) \.  \qb^{\rb(x)+ \rb(y)}.\]
Let \. $\db:=(d_1,\ldots,d_\ell) \in \nn^{\ell}$ \. be an arbitrary integer vector.
Denote by
\[  \Phi_{\db} \ := \  \big[q_1^{d_1} \. \cdots \, q_\ell^{d_\ell}\big] \, \Phi(q)
\]
the coefficient of the monomial \. $\qb^{\db}$ \. in   \. $\Phi(q)$.  We then have:
\[  \Phi_{\db} \ = \   \sum_{\substack{(x,y) \ts \in \ts \cL^2,\\ \rb(x)+ \rb(y)\ts =\ts \db}} \. \phi(x,y).\]

We now consider another, slightly coarser, grouping of terms.
For \ts $u,v \in \cL$ \ts satisfying \ts $u \prec^\di v$, so in particular \. $u \neq v$,
let \ts $C(u,v)$ \ts denote the set of (ordered) pairs \ts $(x,y)$ \ts in the interval \ts $[u,v]$ \ts
such that  \. $x \wedge y= v$ \. and \. $x \vee y = u$.
Let
\[ \psi(u,v) \ := \  \sum_{(x,y)  \ts \in \ts C(u,v)} \. \phi(x,y). \]
It follows from the modularity of \. $r_1,\ldots, r_\ell$ \.  that
\[ \Phi_{\db} \ = \  \sum_{\substack{u \ts \prec^\di \ts v, \\ \. \rb(u)+\rb(v) \ts = \ts\db}} \  \psi(u,v)  \ + \   \sum_{\substack{u  \ts \in \ts  \cL,\\ \. 2 \ts \rb(u) \ts = \ts \db}} \. \phi(u,u).\]
Since \. $\phi(u,u) = \alpha(u) \beta(u) \. - \. \gamma(u)\delta(u) \. \leq \. 0 $ \.  by \eqref{eq:AD-condition},
the proposition follows from Claim~\ref{claim:q-FKG} below. \end{proof}

\smallskip

\begin{claim}\label{claim:q-FKG}
In notation above, for every \ts $u,v\in \cL$ \ts such that \ts $u \prec^\di v$, we have \. $\psi(u,v) \ts \leq  \ts 0$.
\end{claim}

\smallskip

\begin{proof}[Proof of Claim~\ref{claim:q-FKG}]
	Note that \ts $\psi(u,v)$ \ts depends only on elements in the poset interval \ts $[u,v]$\ts,
	so by restricting to \ts $[u,v]$ \ts if necessary,
	we can without loss of generality assume that \ts $u=\widehat{0}$ \ts is the unique minimal element of $\cL$,  and \ts $v=\widehat{1}$ \ts is the unique maximal element of $\cL$.
	
	For \ts $x \in \cL$\ts,  a \defn{complement} of \ts $x$ \ts is an element \ts $y \in \cL$  \ts such that \. $x\wedge y=\widehat{0}$ \. and \. $x \vee y = \widehat{1}$\..
	Note that in a finite distributive lattice every element has at most one complement (see e.g.\ \cite[Thm~10,~p.~12]{Bir}), and we denote this element by $x^c$ if it exists.
	Note that \ts $\psi(\widehat{0},\widehat{1})$ \ts depends only on elements that have a complement in~$\cL$, and that the set of complemented elements in a finite distributive lattice form a sublattice of $\cL$ (see e.g.\ \cite[p.~18]{Bir}).
	By restricting to this sublattice if necessary, without loss of generality we can assume that  every element \ts $x \in \cL$ \ts has a unique complement \ts $x^c$ \ts (i.e., when \ts $\cL$ \ts is a Boolean lattice).
	
	Define four new functions \. $\alpha',\beta',\gamma',\delta' \. : \. \cL \to \Rb_+$ \. as follows:
	\begin{align*}
		\alpha'(x) \ := \ \alpha(x) \. \beta(x^c), \quad
		\beta'(x) \ := \ \alpha(x^c) \. \beta(x), \quad
		\gamma'(x) \ := \ \gamma(x) \. \delta(x^c), \quad
		\delta'(x) \ := \ \gamma(x^c) \. \delta(x).	
	\end{align*}
	Note that
	\[  \psi(\widehat{0},\widehat{1}) \ = \ \sum_{x \in \cL} \phi(x,x^c) \ = \  \sum_{x \in \cL} \alpha(x)  \beta(x^c) \ -  \ \gamma(x)\delta(x^c) \ = \ \alpha'(\cL) \ - \ \gamma'(\cL). \]
	It thus suffices to show that \. $\alpha'(\cL) \leq \gamma'(\cL)$.
Now observe that, for any \ts $x,y \in \cL$, we have:
\begin{align*}
	\alpha'(x) \beta'(y) \ &= \  \big(\alpha(x) \beta(y) \big) \. \big( \alpha(y^c) \beta (x^c) \big) \
	\leq_{\eqref{eq:AD-condition}} \ \big( \gamma(x \vee y) \. \delta(x \wedge y ) \big) \.
	\big( \gamma(y^c \vee x^c) \. \delta(y^c \wedge x^c ) \big) \\
	&\leq \ \gamma(x \vee y) \. \delta((y\vee x)^c) \. \gamma((y \wedge x)^c) \. \delta(x \wedge y) \ = \ \gamma'(x \vee y) \. \delta'(x \wedge y).
\end{align*}
	It then follows from the (usual) AD~inequality \eqref{eq:AD}, that
\begin{equation}\label{eq:abgd}
  \alpha'(\cL) \. \beta'(\cL) \ \leq \ \gamma'(\cL) \. \delta'(\cL).
\end{equation}
On the other hand, note that \. $\beta'(\cL) = \alpha'(\cL) $ \. and \. $\gamma'(\cL) = \delta'(\cL)$ \. by definition of the functions.
%
%
Since the functions are nonnegative, \eqref{eq:abgd} gives \. $\alpha'(\cL) \leq \gamma'(\cL)$.
This completes the proof.
\end{proof}	

\bigskip

\section{Proof of main results}\label{s:proof-main}

%
%
%
%


\subsection{Proof of Theorem~\ref{t:main-OP-multi}}
Let \. $\alpha,\beta,\gamma, \delta: \cL \to \Rb_{+}$ \. be as in \eqref{eq:Fish-ftn}.
Note that these functions satisfy the assumption \eqref{eq:Fish-AD-cond} of the
multivariate AD inequality.

Let $\qb:=(q_1,\ldots,q_n)$ be variables, with $n=|X|$.
	For any $i \in [n]$,
	let \. $r_i:\cL \to \Rb_+$ \.
	be the modular function given by \. $r_i(T) \. := \. T(x_i)$.	
For a subset \ts $Y\subseteq X$, denote
$$\qb^{\mathbf{n}(Y)} \, := \, \prod_{x_i \ts \in \ts Y} \. (q_i)^t\,.
$$
Then:
	\begin{alignat*}{2}
		\alpha_{\lqr}(\cL) \ &= \    \Omega_{\qb}(X-A-C,t) \. \cdot \. \qb^{\mathbf{n}(C)}, \\
		\beta_{\lqr}(\cL)  \ &= \   \Omega_{\qb}(X-B-D,t) \. \cdot \. \qb^{\mathbf{n}(D)},\\
		\gamma_{\lqr}(\cL) \ &= \   \Omega_{\qb}(X-V ,t) \. \cdot \. \qb^{\mathbf{n}(C\cup D)}, \\
		\delta_{\lqr}(\cL) \ &= \    \Omega_{\qb}(X-W,t) \. \cdot \. \qb^{\mathbf{n}(C\cap D)}.
	\end{alignat*}
	The theorem now follows from the multivariate AD inequality \eqref{eq:AD-multi}.
\qed

\smallskip


\smallskip

%

\subsection{Proof of Theorem~\ref{t:main-K}}
		Let \. $\alpha,\beta,\gamma, \delta: \cL \to \Rb_{+}$ \. be as in \eqref{eq:Fish-ftn}, with $t \gets N$.
	Note that these functions satisfy the assumption of the multivariate AD inequality (see~\eqref{eq:Fish-AD-cond}).
	Let $\qb:=(q_0,\ldots,q_N)$ be variables.
	For any $i \in \{0,\ldots,N\}$,
	let \. $r_i:\cL \to \Rb_+$ \.
	be the modular function where \. $r_i(T) \. := \. |\{ x \in X \. : \. T(x)=i \}|$ \. is the number of  $i$'s  in $T$.	
	Then
		\begin{alignat*}{2}
		\alpha_{\lqr}(\cL) \ &= \     \rK_{\bbz}(X-A-C,N) \. \cdot \.  q_0^{|A|} \. q_N^{|C|}\., \\
		\beta_{\lqr}(\cL)  \ &= \    \rK_{\bbz}(X-B-D,M) \. \cdot \. q_0^{|B|} \. q_N^{|D|}\.,\\
		\gamma_{\lqr}(\cL) \ &= \    \rK_{\bbz}(X-V ,N) \. \cdot \. q_0^{|A \cap B|} \. q_N^{|C \cup D|}\., \\
		\delta_{\lqr}(\cL) \ &= \     \rK_{\bbz}(X-W,M) \. \cdot \. q_0^{|A \cup B|} \. q_N^{|C \cap D|}\..
	\end{alignat*}
	The theorem now follows from the multivariate AD inequality \eqref{eq:AD-multi}.
\qed

\medskip

\section{Back to Schur polynomials}\label{s:Schur}

In this section we give a new proof of the Lam--Pylyavskyy inequality \eqref{eq:LP2}
for Schur polynomials via the  multivariate AD inequality.

 \smallskip

\subsection*{Proof of Theorem~\ref{t:P-Schur}}
%
%
	Let \ts $\cP:= \cP_\la$ be the poset of the Young diagram of shape $\lambda$, where \ts $\la:= \mu \vee \nu$.
 	Let \ts $\LL:=(\cL',\wedge',\vee')$ \ts be the distributive lattice given by \. $\cL':=\RPP(\la,N)$\.,  with the $\vee'$ and $\wedge'$ operation given by
 	\[ (S \vee' T)(i,j) \ := \  \max\{S(i,j), T(i,j) \}, \qquad (S \wedge' T)(i,j) \ := \  \min\{S(i,j), T(i,j) \}. \]

For a skew Young diagram \ts $\pi/\tau$ \ts
such that $\pi \subset \lambda$,
let \. $\phi^{\pi/\tau}:\cL' \to \Rb_+$ \.  be the characteristic function of the reverse plane partition $T\in \text{RPP}(\lambda,N)$ satisfying all these properties:
	\begin{align*}
				& T(i.j) \geq 1 \quad \text{ for } (i,j)  \in \la, \\
		&T(i,j) =  1 \quad \text{ for } (i,j) \in \tau \qquad \text{ and } \qquad
		T(i,j) =  N \quad \text{ for } (i,j) \in {\la/\pi},\\
	 		&T(i,j) \ < \ T(i+1,j) \quad \text{if } \ (i,j), \. (i+1,j) \in {\pi/\tau}.
	\end{align*}
Note that these reverse plane partitions are in bijection with semistandard Young tableau of $\pi/\tau$ in $\SSYT(\pi/\tau,N)$.


	We define functions \. $\zeta,\eta,\xi,\rho \ts : \ts \cL \to \Rb_+$ \. as follows:
\begin{align*}
	\ze \ := \ \phi^{\mu/\alpha}, \qquad  \eta \ := \ \phi^{\nu/\beta},
	\qquad \xi \ := \ \phi^{\mu/\al \ts \wedge  \ts \nu/\be}, \qquad
	\rho \ := \ \phi^{\mu/\al \ts \vee \ts \nu/\be}.
\end{align*}
We  now show that these functions satisfy the assumption of the multivariate AD inequality, i.e. for any $S,T \in \cL$:
\begin{align*}
	\ze(S) \.\cdot\. \eta(T) \ \leq \ \xi(S \vee T) \.\cdot\. \rho(S \wedge T),
\end{align*}
The equation is vacuously true if \. $\ze(S)=0$ \. or \. $\eta(T)=0$, so assume \. $\ze(S)=\eta(T)=1$\..
We show only the proof that \. $\xi(S \vee T)=1$\.,
as the proof of  \. $\rho(S \wedge T)=1$ \. is similar.
First, for \ts $(i,j) \in \la$, we have:
\[  [S\vee T](i,j) \ = \  \max\{S(i,j), T(i,j) \} \ \geq  \ 1. \]
Second, for  $(i,j) \in {\alpha \wedge \beta}$,
\[ [S\vee T](i,j) \ = \  \max\{S(i,j), T(i,j) \} \ = \ 1, \]
Third, for $(i,j) \in {\lambda/(\mu \wedge  \nu)}$,
\[ [S\vee T](i,j) \ =  \  \max\{S(i,j), T(i,j) \} \ = \ N. \]
Fourth, let  \. $(i,j), \. (i+1,j) \in {(\mu \wedge \nu)/(\alpha \wedge \beta)}$.
We will need to show that
\begin{equation}\label{eq:Sch-1}
 [S \vee T](i,j) \, < \, [S \vee T](i+1,j).
\end{equation}
Note that we must have either \. $(i,j) \in (\mu\wedge \nu)/\alpha$ \. or
\. $(i,j) \in (\mu \wedge\nu)/\be$\..
Without loss of generality, we assume the former holds.
Then it follows that $(i+1,j) \in (\mu \wedge \nu)/\alpha$.
Since \ts $\ze(S)=1$, this implies that
\[   S(i,j) \ < \  S(i+1,j)  \ \leq \ \max\{S(i+1,j), T(i+1,j)\} \ = \ [S \vee T](i+1,j). \]
Thus \eqref{eq:Sch-1} follows if \. $T(i,j) \leq S(i,j)$\., so suppose instead that \. $T(i,j) > S(i,j)$\..
This then implies \. $T(i,j) > 1$\..
Since \. $\eta(T)=1$\., this  implies that \. $(i,j) \in (\mu \wedge\nu)/\beta $, which in turn implies that \. $(i+1,j) \in (\mu \wedge\nu)/\beta $.
Thus we have:
\[ [S\vee T](i,j) \ = \  T(i,j) \ < \ T(i+1,j) \ \leq \ [S \vee T](i+1,j), \]
which completes the proof of \eqref{eq:Sch-1}.

%
%
%
%

Let \. $\bbz:=(z_1,\ldots, z_N)$ \. be variables, and let
\. $r_i:\cL \to \nn$, \. $i \in [N]$, \. be the modular function defined  as follows:
\. $r_i(T):=m_i(T)$ \ts is the number of $i$'s in $T$.  It then follows that
\begin{alignat*}{2}
	A_{\langle \bbz,\rb\rangle} \ &= \  s_{\mu/\al} \. \cdot \. q_1^{|\alpha|} q_N^{|\lambda|-|\mu|}\., \qquad
	&& B_{\langle \bbz,\rb\rangle} \ = \    s_{\nu/\be} \. \cdot \. q_1^{|\beta|} q_N^{|\lambda|-|\nu|}\.,  \\
		C_{\langle \bbz,\rb\rangle} \ &= \   s_{\mu/\al \ts \wedge \nu/\beta} \. \cdot \. q_1^{|\alpha \wedge \beta|} q_N^{|\lambda|-|\mu\wedge \nu|}\., \qquad
	&& D_{\langle \bbz,\rb\rangle} \ = \ s_{\mu/\al \ts \wedge \nu/\beta} \. \cdot \.  q_1^{|\alpha \vee \beta|} q_N^{|\lambda|-|\mu\vee \nu|}\..
\end{alignat*}
The theorem now follows from the multivariate AD inequality \eqref{eq:AD-multi}.
\qed

\smallskip

\begin{rem}\label{r:LP-gen}
By the arguments analogous to the proofs in this and previous section, specifically
the proof of \eqref{eq:Sch-1} to account for strict comparisons, the multivariate AD inequality can be used
to prove results analogous to Theorem~\ref{t:main-OP} and Theorem~\ref{t:main-K}
for both strict and non-strict \ts $(\cP,\omega)$-partitions (see definitions
in \cite[$\S$3.15.1]{EC}).  Similarly, we can extend out results to the more
general $\mathbb{T}$-labelled $(\cP,O)$ tableaux defined in~\cite{LP07}.
We omit the details for brevity.
\end{rem}


\medskip

\section{Multivariate Daykin--Daykin--Paterson inequality} \label{s:DDP}

\subsection{The DDP inequality} \label{ss:DDP-old}
Let \ts $\cP=(X,\prec)$ \ts be a \defnb{partially ordered set} \ts on \ts
$|X|=n$ \ts elements.  Fix \ts $t\ge 0$ \ts and an element \ts $z \in X$.
For  integer $0\le k \le t$, denote by \. $\PP(\cP,t \ts ;z,k)$ \.
the set of $\cP$-partitions \ts $A\in \PP(\cP,t)$ \ts such that $A(z)=k$.
Let  \. $\Omega(\cP,t \ts ; z,k):= \bigl|\PP(\cP,t \ts ;z,k)\bigr|$ \. be the
number of such $\cP$-partitions.
The following inequality was conjectured by Graham~\cite{Gra}
and proved by Daykin--Daykin--Paterson \cite{DDP}.

\smallskip

\begin{thm}[{\rm \defn{Daykin--Daykin--Paterson inequality}}{}]\label{t:DDP}
	Let \ts $\cP=(X,\prec)$ \ts be a finite poset, let $t\in \nn$, and let \ts $z \in X$.
	Then, for every \ts $0 \le k \le t$, we have:
	\begin{equation}\label{eq:DDP}
        \Omega(\cP,t \ts ;z,k)^2 \ \geq \ \Omega(\cP,t \ts ; z,k-1) \. \cdot \. \Omega(\cP,t \ts ; z,k+1).
	\end{equation}
More generally, for every positive integers \ts $a,b \geq 1$,
\begin{equation}\label{eq:DDP-strong}
	\Omega(\cP,t \ts ;z,k+a) \. \cdot \Omega(\cP,t \ts ;z,k+b) \ \geq \ \Omega(\cP,t \ts ; z,k) \. \cdot \. \Omega(\cP,t \ts ; z,k+a+b).
\end{equation}
\end{thm}

\smallskip

We give a new proof of Theorem~\ref{t:DDP} as an application of the AD~inequality \eqref{eq:AD}.
The proof below sets the stage for the multivariate generalization of the theorem.

\smallskip

\begin{proof}[Proof of Theorem~\ref{t:DDP}]
	We denote by \. $\LL=(\cL,\vee,\wedge)$ \. the distributive lattice on the set \. $\cL$ \. given by
	\[  \cL \ := \ \big\{T: X \to \{-b,-b+1,\ldots,t\} \ \, :  \ \, T(x) \leq T(y) \ \ \forall \. x,y\in X \ \
    \text{s.t.} \, \ x\prec y\big\},
    \]
	the set of order-preserving functions such that \. $-b \leq T(x) \leq t$ \. for every $x \in X$.
	The join and meet operation  are given by
	\begin{align*}
		[S\vee T](x) \ := \ \max\{S(x), T(x)\} \quad \text{and}
		\quad [S \wedge T](x) \ := \    \min\{S(x), T(x)\},
	\end{align*}
for every \ts $x \in X$. It is straightforward to verify that $\LL$ is a distributive lattice.

	Let \. $\alpha,\beta,\gamma, \delta: \cL \to \Rb_{+}$ \. be characteristic function of subsets of \. $\cL$ \. defined as follows:
	\begin{align*}
		\alpha \ &:= \ \mathbf{1}\big\{ T(z)=k \ \text{ and } \ T(x) \geq 0, \ \text{ for all } \ x \in X \big\},\\
		\beta \ &:= \ \mathbf{1}\big\{ T(z)=k+a \ \text{ and } \ T(x) \leq  t-b, \ \text{ for all } \ x \in X \big\},\\
		\gamma \ &:= \ \mathbf{1}\big\{ T(z)=k+a \ \text{ and } \ T(x) \geq 0, \ \text{ for all } \ x \in X \big\},\\
		\delta \ &:= \ \mathbf{1}\big\{ T(z)=k \ \text{ and } \ T(x) \leq  t-b, \ \text{ for all } \ x \in X \big\}.
	\end{align*}
	
	We will now verify the assumption of AD inequality, i.e.\ for every \. $S,T \in \cL$, we have:
	\begin{equation}\label{eq:AD-cond}
	 \alpha(S) \.\cdot \. \beta(T) \ \leq \ \gamma(S \vee T) \.\cdot \. \delta(S \wedge T).
	\end{equation}
	Without loss of generality we can assume that \. $\alpha(S)=\beta(T)=1$.
	Note that
	\[  [S\vee T](z) \ = \ \max\{S(z),T(z)\} \ = \  \max\{k,k+a\} \ = \ k+a.\]
	Also note that, for every \ts $x \in X$,
	\[ [S\vee T](x) \ = \ \max\{S(x),T(x)\} \  \geq \ S(x) \ \geq \ 0. \]
	This shows that \. $\gamma(S \vee T)=1$.   Similarly, note that
	\[  [S\wedge T](z) \ = \ \min\{S(z),T(z)\} \ = \  \min\{k,k+a\} \ = \ k.\]
	Also note that, for every \ts $x \in X$,
	\[ [S\wedge T](x) \ = \ \min\{S(x),T(x)\} \  \leq \ T(x) \ \leq \ t-b. \]
	This shows that \. $\delta(S \wedge T)=1$, and completes the proof of \eqref{eq:AD-cond}.
	
\smallskip

	Now  note that
	\begin{align*}
		\alpha(\cL) \ &= \ \big|\{T\in \cL \ : \   T(z)=k \, \text{ and } \, 0 \leq T(x) \leq t \ \ \forall  \. x \in X \}\big| \ = \  \Omega(\cP,t \ts ;z,k), \quad \text{and}\\
		\gamma(\cL) \ &= \ \big|\{T\in \cL \ : \   T(z)=k+a \ \text{ and } \ 0 \leq T(x) \leq t \ \ \forall  \. x \in X \}\big| \ = \  \Omega(\cP,t \ts ;z,k+a).
	\end{align*}
	Also note that
	\begin{equation}\label{eq:Gra-beta}
		\begin{split}
			 \beta(\cL) \ &= \ \big|\{T\in \cL \ : \  T(z)=k+a \ \text{ and } \ -b \leq  T(x) \leq t-b \ \ \forall  \. x \in X \}\big| \\
			&  = \ \big|\{T'\in \cL \ : \   T'(z)=k+a+b \ \text{ and } \ 0 \leq  T'(x) \leq t \ \ \forall  \. x \in X \}\big| \\
                            & = \  \Omega(\cP,t \ts ;z,k+a+b),
		\end{split}
	\end{equation}
	where the second equality is obtained through the substitution \. $T'(x)\gets T(x)+b$.  Similarly, by the
same substitution we have:
	\begin{equation}\label{eq:Gra-delta}
		\begin{split}
			 \delta(\cL) \ &= \ \big|\{T\in \cL \ : \   T(z)=k \ \text{ and } \ -b \leq  T(x) \leq t-b \ \ \forall  \. x \in X \}\big| \\
			& = \ \big|\{T'\in \cL \ : \   T'(z)=k+b \ \text{ and } \ 0 \leq  T'(x) \leq t \ \ \forall  \. x \in X \}\big| \\
& = \  \Omega(\cP,t \ts ;z,k+b).
		\end{split}
	\end{equation}
Now \eqref{eq:DDP-strong} follows from the AD inequality \eqref{eq:AD}.
\end{proof}

\smallskip

\begin{rem}\label{r:Graham}
The original proof of the DDP inequality was through an explicit injection \cite{DDP}.
Curiously, Graham believed that there should exist a proof based on the FKG or AD~inequalities.
He lamented: ``such a proof has up to now successfully eluded all attempts to find it''
\cite[p.~15]{Gra}.  The proof above validates Graham's supposition.

We should also mention that if the order-preserving functions are replaced with
linear extensions, the DPP inequality \eqref{eq:DDP} becomes \defng{Stanley's inequality}
\cite{Sta}, a major result in the area for which finding a direct combinatorial proof
remans a challenging open problem.  We refer to \cite[$\S$6.3]{Pak-what} for an
extensive discussion and further references.
\end{rem}

\smallskip

\subsection{Multivariate DDP inequality}\label{ss:DDP-multi}
Let \. $\qb :=(q_1,\ldots, q_n)$ \.  be variables, and  fix a natural labeling \.
$X=\{x_1,\ldots, x_n\}$.  Define
 \[  \Omega_{\qb}(\cP,t \ts ;z,k) \ := \ \sum_{A\ts \in \ts \PP(\cP, \ts t \ts ;\ts z,\ts k)} q_1^{A(x_1)}  \cdots \. q_n^{A(x_n)}\..
 \]
We now present the multivariate version of DDP inequality~\eqref{eq:DDP},
proved by the multivariate AD inequality~\eqref{eq:AD-multi}.

\smallskip

\begin{thm}[{\rm \defn{multivariate DDP inequality}}{}]\label{t:DDP-multi}
	Let \ts $\cP=(X,\prec)$ \ts be a finite poset, let $t\in \nn$, and let \ts $z \in X$.
	Then, for every \ts $0 \le k \le t$, we have:
	\begin{equation}\label{eq:DDP-multi}
        \Omega_\qb(\cP,t \ts ;z,k)^2 \ \geqslant_\qb \ \Omega_\qb(\cP,t \ts ; z,k-1) \. \cdot \. \Omega_\qb(\cP,t \ts ; z,k+1).
	\end{equation}
More generally, for every integer \ts $a,b \ge 1$\ts, we have:
	\begin{equation}\label{eq:DDP-multi-gen}
        \Omega_\qb(\cP,t \ts ;z,k+a) \. \cdot \. \Omega_\qb(\cP,t \ts ;z,k+b) \ \geqslant_\qb \ \Omega_\qb(\cP,t \ts ; z,k) \. \cdot \. \Omega_\qb(\cP,t \ts ; z,k+a+b).
	\end{equation}
\end{thm}

\smallskip

Note that in contrast with the DPP inequality~\eqref{eq:DDP}, the generalized
log-concavity \eqref{eq:DDP-multi-gen} does not follow from the (usual)
log-concavity~\eqref{eq:DDP-multi} via telescoping.

\smallskip

\begin{proof}
	Let \. $\cL, \alpha,\beta,\gamma, \delta$ \. be as in the proof of Theorem~\ref{t:DDP}.
	Note that these functions satisfy the assumption \eqref{eq:AD-condition} of the
	multivariate AD inequality~\eqref{eq:AD-multi}.
	%
	For all \ts $1\le i \le n$,
	let \. $r_i:\cL \to \Rb_+$ \.
	be the modular function given by \. $r_i(A) \. := \. A(x_i)$, where \ts $A\in \cL$. 	
	Then:
	\begin{alignat*}{2}
		\alpha_{\lqr}(\cL) \ &= \    \Omega_{\qb}(\cP,t \ts ;z,k), \qquad
	&&	\beta_{\lqr}(\cL)  \  = \   \Omega_{\qb}(\cP,t \ts ;z,k+a+b) \. \cdot \. (q_1\cdots q_n)^{-b},\\
		\gamma_{\lqr}(\cL) \ &= \   \Omega_{\qb}(\cP ,t \ts ;z,k+a), \qquad
	&&	\delta_{\lqr}(\cL) \ = \    \Omega_{\qb}(\cP,t \ts ;z,k+b) \. \cdot \. (q_1\cdots q_n)^{-b}.
	\end{alignat*}
	The second part of the theorem now follows from the multivariate AD inequality \eqref{eq:AD-multi}, and thus also the first part (which is a special case).
%
\end{proof}

\smallskip

\begin{rem}\label{r:DDP-multi}  In the context of Remark~\ref{r:LP-gen},
Theorem~\ref{t:DDP-multi} holds  by the same argument
if the order-preserving functions are replaced
with the strict order-preserving functions.  This approach can be extended
to	general $\mathbb{T}$-labelled $(\cP,O)$ tableaux.
However, the analogue of \eqref{eq:DDP-multi} does not hold if \ts
$\Omega_{\qb}$ \ts is replaced with \ts $\rK_{\bbz}$.
This is because the weight functions for $\rK_{\bbz}$ is not invariant
under the translation transformation used in the equations \eqref{eq:Gra-beta}
and \eqref{eq:Gra-delta} in the proof of Theorem~\ref{t:DDP}.
\end{rem}

\smallskip

\subsection{Log-concavity of the multivariate order polynomial}\label{ss:DPP-OP}
The following corollary follows immediately from Theorem~\ref{t:DDP-multi},
and can be viewed as a multivariate generalization of \cite[Thm~4.7]{CPP},
and a poset generalization of the first formula in the proof of Lemma~6.13
in \cite[p.~550]{LPR}.

\smallskip

\begin{cor}\label{c:DPP-multi}
	Let \ts $\cP=(X,\prec)$ \ts be a finite poset, and let \ts $t\in \nn_{\ge 1}$ \ts be a positive integer.
	Then:
	\[ \Omega_{\qb}(\cP,t)^2 \ \geqslant_\qb \ \Omega_{\qb}(\cP,t-1) \. \cdot \. \Omega_{\qb}(\cP,t+1). \]
More generally, for every integers \ts $a, b \ge 1$, we have:
	\[ \Omega_{\qb}(\cP,t+a) \. \cdot \.  \Omega_{\qb}(\cP,t+b)  \ \geqslant_\qb \ \Omega_{\qb}(\cP,t) \. \cdot \. \Omega_{\qb}(\cP,t+a+b). \]
\end{cor}

\smallskip

\begin{proof}
	Let \. $n:=|X|$.  Let \. $\cP':=\cP \oplus z$ \. be the linear sum of \ts $\cP$ \ts and an extra element~$z$,
which is the unique maximal element in~$\cP'$. Since we use natural labeling, element \ts $z$ corresponds to
the variable \ts  $q_{n+1}$\ts.

Note that for every \ts $\ell,t \in \nn$\ts, we have:
	\begin{align}\label{eq:subs}
		\Omega_{\qb}(\cP',t \ts ;z,\ell) \ &= \ \Omega_{\qb}(\cP,\ell) \. \cdot \. q_{n+1}^\ell.
	\end{align}
On the other hand, it follows from applying Theorem~\ref{t:DDP-multi} to \ts $\cP'$ \ts that
$$
        \Omega_\qb(\cP',t \ts ;z,k+a) \. \cdot \. \Omega_\qb(\cP',t \ts ;z,k+b) \ \geqslant_\qb \ \Omega_\qb(\cP',t \ts ; z,k) \. \cdot \. \Omega_\qb(\cP',t \ts ; z,k+a+b).
$$
The corollary now follows by applying \eqref{eq:subs} to the equation above.
\end{proof}

\smallskip

\begin{rem}\label{r:OP-gen}
Our proof of the \ts $\qb=\textbf{1}$ \ts version in \cite[Thm~4.7]{CPP} goes along similar
lines, but uses the FKG rather than the AD~inequality.  Note that our \cite[Thm~4.8]{CPP}
gives a strict log-concavity for order polynomials, with a substantially more involved proof.
\end{rem}

\medskip

\section{Cross--product inequality for $\cP$-partitions}\label{s:CPC}

\subsection{The statement} \label{ss:CPC-thm}
Let \ts $\cP=(X,\prec)$ \ts be a poset \ts on \ts
$|X|=n$ \ts elements.  Fix \ts $t\ge 0$ \ts and distinct elements \ts $x,y,z \in X$.
For  integers \ts $k,\ell \geq 0$\ts, denote by
$$\SPP(\cP,t \ts ;x,y,z \ts ;k,\ell) \ := \ \big\{A\in \PP(\cP,t) \ : \ A(y)-A(x)=k \ \ \text{and} \ \ A(z)-A(y)=\ell\big\}.
$$
Denote
\[  \Lambda_q(k,\ell) \, := \, \sum_{A\ts \in \ts \SPP(\cP, \ts t \ts ;\ts x,y,z \ts ;\ts k,\ell)}  q^{|A|}\., \qquad
\Lambda_{\qb}(k,\ell) \, := \, \sum_{A\ts \in \ts \SPP(\cP, \ts t \ts ;\ts x,y,z \ts ;\ts k,\ell)} q_1^{A(x_1)}  \cdots \. q_n^{A(x_n)}\.,
\]
and let \. $\aF(k,\ell) \. :=\. \Lambda_1(k,\ell) \. = \. \big|\SPP(\cP,t \ts ;x,y,z \ts ;k,\ell)\big|$.

\smallskip

\begin{thm}[{\rm \defn{Cross-product inequality for $\cP$-partitions}}{}]
\label{thm:CPC-OP}
	Let \ts $\cP=(X,\prec)$ \ts be a finite poset, let $x,y,z \in \cP$, and let \ts $t\in \nn_{\ge 1}$ \ts be a positive integer.
Then, for every \ts $k,\ell \geq 0$\ts, we have:
\begin{equation}\label{eq:CPC-OP-1}
 \aFr(k,\ell+1) \. \cdot \.  \aFr(k+1,\ell) \ \geq \  \aFr(k,\ell) \. \cdot \.  \aFr(k+1,\ell+1).
\end{equation}
More generally:
\begin{equation}\label{eq:CPC-OP-q}
 \Lambda_{q}(k,\ell+1) \. \cdot \. \Lambda_{q}(k+1,\ell) \ \geqslant_q \ \Lambda_{q}(k,\ell) \. \cdot \. \Lambda_{q}(k+1,\ell+1).
\end{equation}
Even more generally:
\begin{equation}\label{eq:CPC-OP}
 \Lambda_{\qb}(k,\ell+1) \. \cdot \. \Lambda_{\qb}(k+1,\ell) \ \geqslant_\qb \ \Lambda_{\qb}(k,\ell) \. \cdot \. \Lambda_{\qb}(k+1,\ell+1).
\end{equation}
\end{thm}

\smallskip

\begin{rem}\label{r:CPC}
Note that already the unweighted inequality \eqref{eq:CPC-OP-1} appears to be new.
Note also that if the order-preserving functions are replaced with linear extensions,
then a version of \eqref{eq:CPC-OP-1} is known as the
\defng{cross--product conjecture}~\cite[Conj~3.1]{BFT}, a major open problem
in the area.  We refer to \cite{CPP22a} for an extensive discussion and
further references.
\end{rem}

\smallskip

\subsection{Proof of Theorem~\ref{thm:CPC-OP}} \label{ss:CPC-proof}
We denote by \. $\LL=(\cL,\vee,\wedge)$ \. the distributive lattice on the set
of order-preserving functions from $X$ to $\{0,1,\ldots,t\}$:
\[  \cL \ := \ \big\{T: X \to \{0,1,\ldots,t\} \ \, :  \ \, T(v) \leq T(w) \ \ \forall \. v,w\in X \ \
\text{s.t.} \, \ v\prec w\big\}.
\]
The join and meet operation  are given by
\begin{align*}
	[S\vee T](w) \ &:= \ \max\big\{S(w)-S(y), \, T(w)-T(y)\big\} \ + \ \min\big\{S(y), T(y)\big\}, \\
	[S\wedge T](w) \ &:= \ \min\big\{S(w)-S(y), \, T(w)-T(y)\big\} \ + \ \max\big\{S(y), T(y)\big\},
\end{align*}
for every \ts $w \in X$.
This lattice was proved distributive by Shepp \cite[Eq.~2.4, 2.5]{She}, in his proof of the \defng{$XYZ$
inequality} (see also \cite[$\S$6.4]{AS}).

Let \. $\alpha,\beta,\gamma, \delta: \cL \to \Rb_{+}$ \. be characteristic function of subsets of \. $\cL$ \. defined as follows:
 \begin{align*}
 	\alpha \ &:= \ \mathbf{1}\big\{ T(y)-T(x)=k \ \text{ and } \ T(z)-T(y) =\ell \big\},\\
 	\beta \ &:= \ \mathbf{1}\big\{ T(y)-T(x)=k+1 \ \text{ and } \ T(z)-T(y) =\ell+1 \big\},\\
 	\gamma \ &:= \ \mathbf{1}\big\{ T(y)-T(x)=k \ \text{ and } \ T(z)-T(y) =\ell+1 \big\},\\
 	\delta \ &:= \ \mathbf{1}\big\{ T(y)-T(x)=k+1 \ \text{ and } \ T(z)-T(y) =\ell \big\}.
 \end{align*}

	We will now verify the assumption \eqref{eq:AD-condition} of the multivariate AD inequality:
\begin{equation}\label{eq:AD-cond-CPC-proof}
	\alpha(S) \.\cdot \. \beta(T) \ \leq \ \gamma(S \vee T) \.\cdot \. \delta(S \wedge T),
\end{equation}
for every \. $S,T \in \cL$. Without loss of generality we can assume that \. $\alpha(S)=\beta(T)=1$.
We have:
\begin{align*}
	[S\vee T](x) \, - \, [S\vee T](y) \ &= \  \max\{S(x)-S(y), \, T(x)-T(y)\} \ = \ \max\{-k, -k-1\} \ = \  -k, \\
	[S\vee T](z) \, - \, [S\vee T](y) \ &= \  \max\{S(z)-S(y), \, T(z)-T(y)\} \ = \ \max\{\ell, \ell+1\} \ = \  \ell+1,\\
	[S\wedge T](x) \, - \, [S\wedge T](y) \ &= \  \min\{S(x)-S(y), \, T(x)-T(y)\} \ =  \ \min\{-k, -k-1\} \ = \  -k-1, \\
	[S\wedge T](z) \, - \, [S\wedge T](y) \ &= \  \min\{S(z)-S(y), \, T(z)-T(y)\} \ = \ \min\{\ell, \ell+1\} \ = \  \ell.
\end{align*}
This shows that \. $\gamma(S \vee T) = \delta(S \wedge T)=1$ and proves \eqref{eq:AD-cond-CPC-proof}.

\smallskip

Finally, consider modular functions \. $r_i:\cL \to \Rb_+$\ts, for all \ts $1\le i \le n$,
given by \. $r_i(T) \. := \. T(x_i)$.  Then we have:
\begin{alignat*}{2}
	\alpha_{\lqr}(\cL) \ &= \    \Lambda_{\qb}(k,\ell), \qquad
	&& \beta_{\lqr}(\cL)  \  = \  \Lambda_{\qb}(k+1,\ell+1),\\
	\gamma_{\lqr}(\cL) \ &= \   \Lambda_{\qb}(k,\ell+1) \qquad
	&& \delta_{\lqr}(\cL) \ = \    \Lambda_{\qb}(k+1,\ell).
\end{alignat*}
The  theorem now follows from the multivariate AD inequality \eqref{eq:AD-multi}.
\qed

\smallskip

\begin{rem}
Let us also mention that the proof in \cite[$\S$3.1]{CPP22a} shows that Theorem~\ref{thm:CPC-OP}
implies a (multivariate) $\cP$-partition version of the \defng{Kahn--Saks inequality} \cite[Thm~2.5]{KS}.
On the other hand, while the KS~inequality easily implies Stanley's inequality discussed earlier in
Remark~\ref{r:Graham} (see e.g.\ \cite[$\S$1.2]{CPP21}),
the multivariate DPP~inequality (Theorem~\ref{t:DDP-multi}) does not similarly follow from
cross--product inequality for $\cP$-partitions (Theorem~\ref{thm:CPC-OP}).  This is also
demonstrated by the fact that different lattices are used in the proofs of the two theorems.
\end{rem}

\medskip

\section{Final remarks and open problems} \label{s:finrem}

\subsection{} \label{ss:finrem-origin}
This paper grew out of \cite[$\S$4.1]{CP3} where we obtained
superficially similar correlation inequalities which appear to
have a very different nature and whose only known proof uses
the combinatorial atlas technology.  Our investigation was also
partly motivated by the desire to bridge the gap between the
two areas of combinatorics.  Notably, we would like to emphasize
the importance of the AD~inequality to algebraic combinatorics,
and the multivariate weighting to poset theory.

Note that there is a weighted version of \ts $e(\cP)$ \ts introduced
in \cite[$\S$1.16]{CP1}.  While the results in \cite{CP3}
translate verbatim to the weighted setting, these weights seem
incompatible with $\qb$-weights in this paper.
Similarly, the $q$-weight on \ts $e(\cP)$ \ts in \cite{CPP22a}
is also of different nature.  On the other hand, the $q$-weighted
order polynomial in \cite{CPP} is exactly \ts $\Om_q(\cP,t)$.

\subsection{} \label{ss:finrem-SP}
One distinguishing feature of poset inequalities is the difficulty
of getting the equality conditions, see e.g.\ \cite[$\S$9.9]{CPP}
for an overview.  We are not aware of any equality conditions for
the inequalities in this paper, proved or conjectured.

Another difficulty is finding a combinatorial interpretation for
the difference of two sides.  This was a major motivation for
our investigation in \cite{CP1}.  We show in \cite[$\S$7.4]{IP}
that the AD inequality \eqref{eq:AD} does not have a combinatorial
interpretation in full generality, in a sense of being in~$\SP$.
Of course, the Lam--Postnikov--Pylyavskyy deep algebraic approach in \cite{LPP07}
(see Remark~\ref{r:LPP}) is even less likely to give a combinatorial
interpretation. We refer to \cite[$\S$6]{Pak-what} for an extensive survey.

Now, the Lam--Pylyavskyy's injective approach in \cite{LP07} shows
that the difference of coefficients on both sides in \eqref{eq:LP-RPP3}
has a combinatorial interpretation.  By contrast, the limit
arguments we use throughout this paper do not give a combinatorial
interpretation for Fishburn's inequality~\eqref{eq:Fish}.  It would
be interesting to see if \eqref{eq:Fish} and the generalized
Fishburn inequality \eqref{eq:main-LE} can be proved by a direct
combinatorial argument giving a combinatorial interpretation.

\vskip.4cm

\subsection*{Acknowledgements}
We are grateful to Thomas Lam, Greta Panova, Pasha Pylyavskyy and
Yair Shenfeld for helpful discussions and remarks on the subject.
The first author was partially supported by the Simons Foundation.
The second author was partially supported by the~NSF.

\vskip.6cm


\end{document}